\documentclass{article}

\usepackage{amsfonts, amsmath, amssymb, amsthm, amscd}
\usepackage{verbatim}
\usepackage[dvips]{graphicx}
\usepackage{graphics}
\usepackage[all,2cell]{xypic}
\objectmargin+{1mm}
\labelmargin+{0.8mm}
\SelectTips{cm}{12}

\theoremstyle{plain}  
\newtheorem{cor}{Corollary}[section]
\newtheorem{lem}[cor]{Lemma}
\newtheorem{prop}[cor]{Proposition}
\newtheorem{thm}[cor]{Theorem}

\theoremstyle{definition}
\newtheorem{df}[cor]{Definition}
\newtheorem{ex}[cor]{Example}

\newtheorem{para}[cor]{}
\newtheorem{prob}[cor]{Problem}
\newtheorem{rem}[cor]{Remark}

\theoremstyle{remark}
\newtheorem*{claim}{Claim}
\newtheorem*{conv}{Convention}

\newcommand{\bbA}{\mathbb{A}}
\newcommand{\bbK}{\mathbb{K}}
\newcommand{\bbN}{\mathbb{N}}

\newcommand{\bbP}{\mathbb{P}}

\newcommand{\bbZ}{\mathbb{Z}}

\newcommand{\fA}{\mathfrak{A}}

\newcommand{\fe}{\mathfrak{e}}

\newcommand{\ii}{\mathfrak{i}}
\newcommand{\jj}{\mathfrak{j}}
\newcommand{\frakL}{\mathfrak{L}}
\newcommand{\fN}{\mathfrak{N}}

\newcommand{\fX}{\mathfrak{X}}
\newcommand{\fx}{\mathfrak{x}}

\newcommand{\fy}{\mathfrak{y}}

\newcommand{\cA}{\mathcal{A}}
\newcommand{\cB}{\mathcal{B}}
\newcommand{\cC}{\mathcal{C}}
\newcommand{\calD}{\mathcal{D}}
\newcommand{\cE}{\mathcal{E}}
\newcommand{\cF}{\mathcal{F}}

\newcommand{\cI}{\mathcal{I}}

\newcommand{\cM}{\mathcal{M}}

\newcommand{\cP}{\mathcal{P}}

\newcommand{\calR}{\mathcal{R}}
\newcommand{\cS}{\mathcal{S}}
\newcommand{\cT}{\mathcal{T}}

\newcommand{\cX}{\mathcal{X}}
\newcommand{\cY}{\mathcal{Y}}

\newcommand{\bE}{\mathbf{E}}
\newcommand{\bF}{\mathbf{F}}
\newcommand{\bG}{\mathbf{G}}

\newcommand{\Ab}{\operatorname{\bf Ab}}
\newcommand{\AbCat}{\operatorname{\bf AbCat}}

\newcommand{\Bl}{\operatorname{Bl}}

\newcommand{\Cat}{\operatorname{\bf Cat}}
\newcommand{\cf}{\textrm{cf.}\;}
\newcommand{\Ch}{\operatorname{\bf Ch}}

\newcommand{\Coh}{\operatorname{\bf Coh}}
\newcommand{\coker}{\operatorname{Coker}}
\newcommand{\colim}{\operatorname{colim}}

\newcommand{\End}{\operatorname{\bf End}}

\newcommand{\ExCat}{\operatorname{\bf ExCat}}
\newcommand{\Ext}{\operatorname{Ext}}

\newcommand{\fingen}{\operatorname{fg}}

\newcommand{\gr}{\operatorname{gr}}

\newcommand{\hdim}{\operatorname{hdim}}

\newcommand{\Hom}{\operatorname{Hom}}
\newcommand{\HOM}{\mathcal{HOM}}
\newcommand{\Homo}{\operatorname{H}}

\newcommand{\id}{\operatorname{id}}
\newcommand{\im}{\operatorname{Im}}
\newcommand{\isoto}{\overset{\scriptstyle{\sim}}{\to}}

\newcommand{\Kos}{\operatorname{Kos}}

\newcommand{\Lex}{\operatorname{\bf Lex}}

\newcommand{\linf}{\leftarrowtail}

\newcommand{\Mod}{\operatorname{\bf Mod}}

\newcommand{\nil}{\operatorname{nil}}

\newcommand{\onto}[1]{\stackrel{#1}{\to}}
\newcommand{\op}{\operatorname{op}}

\newcommand{\qis}{\operatorname{qis}}

\newcommand{\rdef}{\twoheadrightarrow}
\newcommand{\RelEx}{\operatorname{\bf RelEx}}
\newcommand{\rinc}{\hookrightarrow}
\newcommand{\rinf}{\rightarrowtail}

\newcommand{\Set}{\operatorname{\bf Set}}

\newcommand{\ssm}{\smallsetminus}
\newcommand{\Supp}{\operatorname{Supp}}

\newcommand{\tf}{\operatorname{tf}}

\UseAllTwocells

\def\sn{\smallskip\noindent}
\def\mn{\medskip\noindent}

\title{Homotopy invariance of higher $K$-theory for 
abelian categories}
\author{Satoshi Mochizuki \and 
Akiyoshi Sannai}
\date{}

\begin{document}
\maketitle

\begin{abstract}
The main theorem in this paper is that 
the base change functor from a noetherian abelian category $\cA$ to 
$\cA[t]$ the noetherian polynomial category of $\cA$, 
$-\otimes_{\cA}\bbZ[t]\colon\cA \to \cA[t]$ 
induces an isomorphism on $K$-theory. 
The main theorem implies the well-known fact 
that $\mathbb{A}^1$-homotopy invariance of $K'$-theory for noetherian schemes.

\sn
\footnotesize{{\bf Classification} 18E10, 19D35.} 
\end{abstract}

\section{Introduction}
Contrary to the importance of $\mathbb{A}^1$-homotopy invariance 
in the motivic homotopy theory \cite{Voe98}, \cite{MV99} and \cite{Voe00}, 
the homotopy invariance of $K'$-theory for noetherian schemes still 
has been mysterious in the following sense. 
Recall the footstep of $K$-theory 
in the viewpoint of axiomatic characterization. 
The observation that the additivity theorem 
is the fundamental theorem of connective algebraic $K$-theory 
is implicit in Waldhausen \cite{Wal85} (See also \cite{GSVW92}) 
and was known to Grayson \cite{Gra87}, Staffeldt \cite{Sta89} and 
McCarthy \cite{McC93}. 
Recently the connective $K$-theory is regarded as 
the universal additive invariant 
by Tabuada \cite{Tab08} and Barwick \cite{Bar14}. 
After Thomason \cite{TT90}, 
derived invariance and localizing properties of non-connective $K$-theory 
are emphasized by many authors \cite{Nee92}, \cite{Kel99}, 
\cite{TV04}, \cite{Cis10}, \cite{BM11}, \cite{Sch06}, \cite{Sch11} and 
\cite{Moc13}. 
In the landscape of non-commutative motive theory \cite{CT11} 
or motive theory for $\infty$-categories 
\cite{BGT13}, 
non-connective $K$-theory is the universal localizing invariant. 
To relate motivic homotopy theory with motive theory 
for DG or $\infty$-categories, 
it is important to make clear 
that what additional axiom implies 
the homotopy invariance property. 
Many authors have already defined and studied 
affine lines over certain categories as in 
\cite{Alm74}, \cite{Alm78}, 
\cite{Gra77}, \cite{Haz83}, \cite{GM96}, \cite{Yao96} and \cite{Sch06}. 
(See also \cite{Tab14}). 
The main objective in this paper is 
to examine the homotopy invariance of 
$K$-theory for abelian categories by taking Schlichting polynomial categories. 
We recall the definition of the polynomial categories. 
For a category $\cC$, 
we let $\End \cC$ denote 
the {\it category of endomorphisms} 
in $\cC$. 
Namely, an object in $\End\cC$ 
is a pair $(x,\phi)$ 
consisting of an object $x$ in $\cC$ 
and a morphism $\phi:x\to x$ in $\cC$ 
and a morphism 
between $(x,\phi) \to (y,\psi)$ 
is a morphism 
$f:x \to y$ in $\cC$ such that 
$\psi f= f\phi$. 
(See Notation~\ref{nt:endcat}). 
From now on, let $\cA$ be an abelian category. 
We write $\Lex \cA$ for the category of left exact functors from 
$\cA^{\op}$ to $\Ab$ the category of abelian groups. 
The category $\Lex \cA$ is a Grothendieck abelian category and 
the Yoneda embedding $y:\cA \to \Lex \cA$ is exact and reflects 
exactness. 
We say an object $x$ in $\cA$ is 
{\it noetherian} 
if every ascending filtration of subobjects of $x$ is stational. 
We say $\cA$ is {\it noetherian} if every object in $\cA$ is noetherian. 
(See Definition~\ref{nt:noetherianobj}). 
We assume that $\cA$ is a noetherian abelian category and 
we write $\cA[t]$ for the full subcategory of 
noetherian objects 
in $\End \Lex\cA$ 
and 
call it the {\it noetherian polynomial category} 
over $\cA$. 
(See Definition~\ref{df:S-poly cat def}). 
We can prove that $\cA[t]$ is an abelian category. 
(See Lemma~\ref{lem:noetherian}). 
For an object $a$ in $\cA$, 
let us define an object $a[t](=(a[t],t))$ in $\End \Lex\cA$ as follows. 
The underlying object $a[t]$ is 
$\displaystyle{\bigoplus_{n=0}^{\infty}}at^i$ 
where $at^i$ is a copy of $a$. 
The endomorphism $t:a[t] \to a[t]$ is defined by 
the identity morphisms $at^i \to at^{i+1}$ in each components. 
We can prove that if $a$ is noetherian in $\cA$, then 
$a[t]$ is noetherian in $\cA[t]$. (See Theorem~\ref{thm:Abst Hilb basis}). 
We call the association $-\otimes_{\cA}\bbZ[t]:\cA \to \cA[t]$, 
$a \mapsto a[t]$ the {\it base change functor} which is an exact functor. 
One of the consequence of the main theorem is the following:

\begin{thm}[\bf A part of Theorem~\ref{thm:abstract main thm}]
\label{thm:mainthm} 
Let $\cA$ be a noetherian abelian category. 
The functor $-\otimes_{\cA} \bbZ[t]:\cA \to \cA[t]$ induces 
a homotopy invariance of spectra on $K$-theory
$$K(\cA)\isoto K(\cA[t]).$$
\end{thm}

The key idea of how to prove the main theorem is, 
roughly speaking, 
that we recognize an affine space as to be a rudimental projective space
$$\bbA^n=\bbP^n\ssm\{[x_0:\cdots:x_n]\in\bbP^n;x_n=0 \}(=\bbP^n\ssm\bbP^{n-1}).$$
(Compare the equation above with the formula $\mathrm{(2)}$ below). 
To give more precise explanation, 
for a scheme $X$ which has an ample family of line bundles and 
a closed subset $Y$ of $X$, 
we write $[X^Y]$ and $K(\Coh_YX)$ for the derived category and the 
non-connective $K$-theory 
of bounded complexes of coherent sheaves $E^{\bullet}$ on $X$ such that 
$\displaystyle{\underset{i}{\bigcup} \Supp \Homo^i(E^{\bullet}) \subset Y}$ 
respectively 
and denote $[X^X]$ and $K(\Coh_X X)$ by $[X]$ and $K(\Coh X)$ respectively. 
Then the following three formulas imply $\bbA^1$-homotopy invariance 
of $K$-theory of coherent sheaves over noetherian schemes:\\
$\mathrm{(1)}$ 
{\bf (Derived projective bundle formula).} 
$[\bbP^n_X]/[\bbP^{n-1}_X]\isoto [X]$.\\
$\mathrm{(2)}$ 
{\bf (Localization formula).} 
${[\bbP^n_X]/[{(\bbP^n_X)}^{\bbP^{n-1}_X}] } \isoto [\bbA^n_X]$.\\
$\mathrm{(3)}$ 
{\bf (Purity).} 
We have the isomorphism 
$$K(\Coh_{\bbP^{n-1}_X}\bbP^n_X)\isoto K(\Coh \bbP^{n-1}_X).$$
In this paper, we trace parallel argumetns above in categorical setting. 
Projective spaces are replaced with graded categories over categories which 
is introduced in \S 3. 
The formulas $\mathrm{(1)}$ and $\mathrm{(2)}$ above correspond to 
Theorem~\ref{cor:canisomoftf} and 
Theorem~\ref{prop:cannonical isom} respectively. 
Finally the formula $\mathrm{(3)}$ above is replaced with 
Proposition~\ref{prop;devissage} which is 
a consequence of the d{\'e}vissage theorem. 
A geometric meaning of the d{\'e}vissage theorem in the view of 
categorical algebraic geometry 
based on the support varieties theory as in 
\cite{Bal07}, \cite{BKS07} and \cite{Gar09} 
will be studied 
in the first author's subsequent papers. 
See \S \ref{subsec:nilp inv th} for 
an ad hoc axiomization of d{\'e}vissage property. 
In the final subsection, we propose a generalized Vorst problem. 
Recall that 
Vorst conjecture in \cite{Vor79} which says 
that for any affine scheme $X$, 
$\bbA^1$-homotopy invariance of $K$-theory for $X$, 
characterizes the regularilty of $X$, 
has been recently proved in \cite{CHW08} and \cite{GH12}. 
To attack generalized Vorst conjecture, 
the first author hope to extends the arguments in Ibid to 
categorical algebraic geometry setting. 

\begin{conv} 
In this note, 
basically we follow the notation of 
exact categories for \cite{Kel90} and 
algebraic $K$-theory for 
\cite{Qui73} and \cite{Wal85}. 
For example, we call admissible monomorphisms 
(resp. 
admissible epimorphisms and admissible short exact sequences) 
inflations (resp. deflations, conflations). 
We also call a category with cofibrations and weak equivalences 
a Waldhausen category. 
Let us denote the set of all natural numbers by $\bbN$. 
We regard it as a totally ordered set 
with the usual order. 
For a Waldhausen category, 
we denote the specific zero object 
by the same letter $\ast$. 
We denote the $2$-category of essentially small categories by $\Cat$, 
the category of sets by $\Set$ and
the category of essentially small abelian (resp. exact) categories 
by $\AbCat$ (resp. $\ExCat$). 
For any non-negative integer $n$, 
we denote the set of all integers 
$k$ such that $0\leq k \leq n$ by $[n]$. 
For categories $\cX$, $\cY$, 
we denote the (large) category of 
functors from $\cX$ to $\cY$ by 
$\HOM(\cX,\cY)$. 
For any ring with unit $A$, 
we denote the category of 
right $A$-modules (resp. finitely generated right $A$-modules) 
by $\Mod(A)$ (resp. $\cM_A$). 
Throughout the paper, 
we use the letter $\cA$ to denote an essentially small abelian category. 
For an object $x$ in $\cA$ and a finite family 
$\{x_i\}_{1\leq i\leq m}$ 
of subobjects of $x$, $\sum_{i=1}^m x_i$ means 
the minimum subobject of $x$ which contains all $x_i$. 
For an additive category $\cB$, 
we write $\Ch(\cB)$ for 
the category of chain complexes on $\cB$. 
\end{conv}

\section{Polynomial categories}
\label{sec:polycat}

In this section, we recall the notation of 
polynomial abelian categories from \cite{Sch00} or \cite{Sch06}.

\subsection{End categories}

\begin{df}
\label{nt:endcat}
For a category $\cC$, 
we denote 
the {\it category of endomorphisms} 
in $\cC$ by $\End \cC$. 
Namely, an object in $\End\cC$ 
is a pair $(x,\phi)$ 
consisting of an object $x$ in $\cC$ 
and a morphism $\phi:x\to x$ in $\cC$ 
and a morphism 
between $(x,\phi) \to (y,\psi)$ 
is a morphism 
$f:x \to y$ in $\cC$ such that 
$\psi f= f\phi$. 
For any functor $F:\cC \to \cC'$, 
we have a functor
$\End F:\End \cC \to \End \cC'$
which sends 
$(x,\phi)$ to $(Fx,F\phi)$. 
Moreover for any 
natural transformation $\theta:F\to F'$ 
between functors $F$, $F:\cC \to \cC'$, 
we have a natural transformation 
$\End \theta:\End F \to \End F'$ 
defined by the formula 
$\End \theta (x,\phi):=\theta(x)$ 
for any object $(x,\phi)$ in $\End \cC$. 
This association gives a $2$-functor
$$\End:\Cat \to \Cat.$$
We have natural transformations 
$i:\id_{\Cat} \to \End$ and 
$U:\End \to \id_{\Cat}$ defined by 
$i(\cC):\cC \to \End \cC$, 
$x\mapsto (x,\id_x)$ and 
$U(\cC):\End \cC \to \cC$, 
$(x,\phi) \mapsto x$ for each category $\cC$.  
\end{df}

\begin{rem}
\label{rem:limitinendcat}
Let $\cC$ be a category and 
$F:\cI \to \End \cC$, 
$i\mapsto (x_i,\phi_i)$ be a functor. 
Let us assume that 
there is a limit $\lim x_i$ 
(resp. colimit $\colim x_i$) in $\cC$. 
Then we have $\lim F_i=(\lim x_i,\lim \phi_i)$ 
(resp. $\colim F_i=(\colim x_i,\colim \phi_i)$). 
In particular, 
if $\cC$ is additive (resp. abelian), 
then $\End \cC$ is also additive (resp. abelian). 
Moreover if $\cC$ is an exact category 
(resp. a category with cofibration), 
then $\End \cC$ naturally becomes an exact category 
(resp. a category with cofibration). 
Here a sequence 
$(x,\phi) \to (y,\psi) \to (z,\xi)$ is a conflation
if and only if $x \to y \to z$ is a conflation in $\cC$. 
(resp. a morphism $(x,\phi) \onto{u} (y,\psi)$ is a cofibration 
if and only if 
$u:x\to y$ is a cofibration in $\cC$.) 
Moreover if $w$ is 
a class of morphisms in $\cC$ 
which satisfies the axioms of 
Waldhausen categories 
(and its dual), 
then the class of all morphisms 
in $\End \cC$ 
which is in $w$ also satisfies 
the axioms of 
Waldhausen categories (and its dual). 
\end{rem}

\begin{rem}
\label{rem:GMnotation}
In \cite[III. 5.15]{GM96}, 
for a category $\cC$, 
the category $\End \cC$ is called 
the {\it polynomial category} over $\cC$ 
and denoted by $\cC[T]$. 
For any ring with unit $A$, 
we have the canonical category isomorphism 
$$\Mod(A[T]) \isoto (\Mod(A))[T],\ \ M \mapsto (M,T)$$
where $A[T]$ is 
the polynomial ring over $A$ 
and $T$ means an endomorphism 
$T:M \to M$ which sends an element $x$ in $M$ to an element $xT$ in $M$. 
Moreover in general 
for any abelian category $\cA$, 
we have the equality 
$$\hdim \cA[T]=\hdim \cA +1$$
where $\hdim \cA$ is the 
{\it homological dimension} 
of $\cA$ which is defined by 
$$\hdim \cA:=\max\{n; \Ext^{n}(x,y)\neq 0
 \ \ \ \text{for any objects $x$, $y$}\}.$$
But obviously for any right noetherian ring $A$, 
$(\cM_A)[T]$ and $\cM_{A[T]}$ are different categories. 
The main reasons is that 
$A[T]$ is not finitely generated as an $A$-module. 
In particular, 
the object $(A[T],T)$ is in $(\Mod(A))[T]$ 
but not in $(\cM_A)[T]$. 
In the subsection~\ref{subsec:Sch poly cat}, 
we define the noetherian polynomial categories 
over noetherian abelian category which is introduced by 
Schlichting in \cite{Sch06}. 
In this notion, 
we have the canonical category equivalence 
between $\cM_{A[t]}$ and $(\cM_A)[t]$. 
See Example~\ref{ex:S-poly}.
\end{rem}

\subsection{Noetherian objects}
\label{subsec:Noeobj}

In this subsection, 
we develop 
the theory of noetherian objects 
in exact categories 
which is slightly different from 
the usual notation in the category theory.

\begin{df}
\label{nt:noetherianobj}
Let $\cE$ be an exact category 
and $x$ an object in $\cE$. 
We say $x$ is a {\it noetherian object} if 
any ascending filtration of admissible subobjects of $x$ 
$$x_0\rinf x_1 \rinf x_2 \rinf \cdots$$
is stational. 
We say $\cE$ is a {\it noetherian} category if all objects 
in $\cE$ are noetherian. 
\end{df}

We can easily prove the following lemmata.

\begin{lem}
\label{lem:noetherian}
Let $\cE$ be an exact category. 
Then\\
$\mathrm{(1)}$ 
Let $x \rinf y \rdef z$ 
be a conflation in $\cE$. 
If $y$ is noetherian, then $x$ and $z$ are also noetherian.\\
$\mathrm{(2)}$ 
For noetherian objects $x$, $y$ in $\cE$, 
$x\oplus y$ is also noetherian.\\
$\mathrm{(3)}$ 
Moreover assume that $\cE$ is abelian, 
then the converse of 
$\mathrm{(1)}$ is true. 
Namely, in the notation $\mathrm{(1)}$, 
if $x$ and $z$ are noetherian, then $y$ is also noetherian.
\qed
\end{lem}

\begin{lem}
\label{lem:faithfulexact}
For any exact faithful functor $F:\cA \to \cB$ between 
abelian categories and an object $x$ in $\cA$, 
if $Fx$ is noetherian, then $x$ is also noetherian. 
\qed
\end{lem}

\subsection{Grothendieck category}
\label{subsec:Groth cat}

In this subsection, we briefly review the notion of 
Grothendieck categories. 

\begin{df}[\bf Generator]
\label{df:Generator}
An object $u$ in a category $\cC$ is said to be 
a {\it generator} if 
the corepresentable functor $\Hom(u,-):\cC \to \Set$ 
associated with $u$ is faithful. 
\end{df}

\begin{df}[\bf finite type] 
\label{df:finite type} 
Let $\cB$ be an additive category and $x$, $y$ objects in $\cB$. 
We say that $y$ is {\it of $x$-finite type} ({\it in $\cB$}) if 
there exists a positive integer $n$ and an epimorphism 
$x^{\oplus n} \rdef y$ in $\cB$. 
\end{df}

\begin{ex}
\label{ex:finite type} 
Let $R$ be a ring with unit. 
An object $M$ in $\Mod(R)$ is a finitely generated $R$-module 
if and only if $M$ is of $R$-finite type. 
\end{ex}

\begin{lem}
\label{lem:generator}
$\mathrm{(1)}$ 
Let $f:\cB \to \cC$ be an exact functor from an abelian category $\cB$ 
to an exact category $\cC$ 
and $x$, $y$ objects in $\cB$. 
If $y$ is of $x$-finite type, 
then $f(y)$ is of $f(x)$-finite type.\\
$\mathrm{(2)}$ 
Let $\cB$ be an abelian category which has an generator $u$. 
Then any noetherian objects in $\cB$ are of $u$-finite type. 
\end{lem}

\begin{proof}
$\mathrm{(1)}$ 
There exists a positive integer $n$ and an epimorphim 
$p:x^{\oplus n} \rdef y$. 
Then we have an epimorphism $f(p):f(x)^{\oplus n} \rdef f(y)$. 
Hence $f(y)$ is of $f(x)$-finite type. 

\sn 
$\mathrm{(2)}$ 
Let $x$ be a noetherian object in $\cB$ and we put 
$\Lambda:=\Hom(u,x)$. 
For any $\lambda \in \Lambda$, 
we write $u_\lambda$ for a copy of $u$. 
Then $\{\lambda:u_\lambda \to x\}_{\lambda\in \Lambda}$ 
induces a morphsim 
$\displaystyle{p:\bigoplus_{\lambda\in\Lambda}u_{\lambda} \to x}$. 
\begin{claim}
$p$ is an epimorphism. 
\end{claim}
\begin{proof}[Proof of claim]
Let $\alpha:x \to y$ be a non-zero morphism in $\cB$. 
Since $u$ is a generator, 
$\Hom(u,\alpha)$ is a non-zero map. 
Therefore there exists a morphism $\lambda_0:u_{\lambda_0} \to x$ 
such that $\alpha\lambda_0\neq 0$. 
In particular $\alpha p\neq 0$ and $p$ is an epimorphism. 
\end{proof}

\sn
If $\Lambda$ is a finite set, 
then we get the desired result. 
If $\Lambda$ is an infinite set, then 
there exists an injection $\omega:\bbN \to \Lambda$. 
We put $\displaystyle{x_n=p(\bigoplus_{\alpha\in \omega([n])} u_{\alpha})}$ 
where $[n]$ is the set $\{0,1,\cdots,n\}$. 
Then the family $\{x_n\}_{n\in\bbN}$ is an asscending chain of subobjects 
of a noetherian object $x$ and therefore it is stational. 
Say $x_k=x_{k+1}=\cdots$. 
Then the restriction of $p$ to 
$\displaystyle{\bigoplus_{\alpha\in\omega([k])}u_{\alpha}}$, 
$\displaystyle{\bigoplus_{\alpha\in\omega([k])}u_{\alpha} \to x}$ 
is an epimorphism. 
\end{proof}

\begin{df}[\bf Grothendieck category] 
\label{df:Grothendieck category} 
We say that an abelian category $\cB$ is {\it Grothendieck} if 
the following conditions hold.\\
$\mathrm{(1)}$ 
$\cB$ has a generator.\\
$\mathrm{(2)}$ 
$\cB$ is {\it cocomplete}. 
Namely for any small category $\cI$, 
we define the diagonal functor $\Delta_{\cI}:\cB \to \HOM(\cI,\cB)$ 
by sending an object $x$ in $\cB$ to a constant functor 
$\cI \to \cB$ which sends all objects in $\cI$ to $x$ and  
all morphisms in $\cI$ to $\id_x$. 
Then $\Delta_{\cI}$ admits a left adjoint functor 
$\colim_{\cI}:\HOM(\cI,\cB) \to \cB$.\\
$\mathrm{(3)}$ 
All small direct limits in $\cB$ is exact. 
Nameley for any filtered small category $\cI$, 
the colimit functor $\colim_{\cI}:\HOM(\cI,\cB) \to \cB$ is exact. 
\end{df}

\begin{para}
\label{para:GQ-emmbeding}
For an essentially small exact category $\cE$, 
we denote the category of left exact functors from 
$\cE^{\op}$ to the category of abelian groups $\Ab$ by 
$\Lex \cE$. 
It is well-known that 
the category $\Lex \cE$ is a Grothendieck category and 
the Yoneda embedding $y:\cE \to \Lex \cE$ 
which sends $x$ to the representable functor associated with $x$, 
$\Hom(-,x):\cE^{\op} \to \Ab$ is exact and reflects 
exactness. (cf. \cite[A.7.1, A.7.5]{TT90}). 
For example, 
let $A$ be a ring with unit, 
then the composition of the Yoneda embedding 
$\Mod(A) \to \Lex \Mod(A)$ and the restriction 
$\Lex \Mod(A) \to \Lex \cM_A$ induced from the 
inclusion functor $\cM_A \rinc \Mod(A)$ 
is an equivalence 
$$\Mod(A)\isoto \Lex \cM_A$$
where the inverse functor is given by sending an object $F$ in $\Lex \cM_A$ 
to an object $F(A)$ in $\Mod(A)$. 
\end{para}

\begin{thm}[\bf Embedding theorem] 
\label{thm:embedd thm}
{\rm (cf. \cite{GP64}).} 
Let $\cB$ be a Grothendieck category with a generator $u$. 
We put $R:=\Hom_{\cB}(u,u)$. $R$ is a ring with unit by 
taking multiplication as composition of morphisms. 
Then the corepresentable functor 
$\Hom(u,-):\cB \to \Mod-R$ 
associated with $u$ 
is fully faithful. 
\qed 
\end{thm}

\begin{cor}
\label{cor:embedd thm}
Let $\cA$ be an essentially small noetherian abelian category. 
Then there exists a ring with unit $R_{\cA}$ and 
an exact fully faithful functor $i_{\cA}:\cA \to \cM_{R_{\cA}}$. 
\end{cor}

\begin{proof}
Let $u$ be a generator of $\Lex \cA$ and put 
$R_{\cA}:=\Hom(u,u)$. 
Then we have an exact fully faithful functor 
$\bar{i}_{\cA}:\cA \rinc \Mod(R_{\cA})$ defined by composing 
a corepresentable functor asssociated with $u$, 
$\Hom(u,-):\Lex \cA \rinc \Mod(R_{\cA})$ and the Yoneda embedding 
$y_{\cA}:\cA \rinc \Lex\cA$. 
We claim that $\bar{i}_{\cA}$ factors through 
$\cA \rinc \cM_{R_{\cA}}$. 
For any object $x$ in $\cA$, 
$y_{\cA}(x)$ is a noetherian object 
by \cite[5.8.8, 5.8.9]{Pop73}. 
Therefore by Lemma~\ref{lem:generator} $\mathrm{(2)}$, 
$y_{\cA}(x)$ is of $u$-finite type and hence 
$\bar{i}_{\cA}(x)$ is a finitely generated $R_{\cA}$-module 
by Example~\ref{ex:finite type} 
and Lemma~\ref{lem:generator} $\mathrm{(1)}$. 
We obtain the desired result. 
\end{proof}

\subsection{Schlichting polynomial category}
\label{subsec:Sch poly cat}

In this subsection, 
we introduce noetherian polynomial categories 
for noetherian abelian categories.

\begin{para}
\label{para:polynomial category}
For an object $a$ in an additive category with countable coproducts $\cB$, 
we define an object $a[t](=(a[t],t))$ in $\End \cB$ as follows. 
The underlying object $a[t]$ is 
$\displaystyle{\bigoplus_{n=0}^{\infty}}at^i$ 
where $at^i$ is a copy of $a$. 
The endomorphism $t:a[t] \to a[t]$ is defined by 
the identity morphisms $at^i \to at^{i+1}$ in each components. 
We call the object $a[t]$ in $\End\cB$ 
the {\it polynomial object of $a$}. 
For an object $a$ in an essentially small exact category $\cE$, 
we similarly define an object $a[t]$ in $\End \Lex\cE$. 
\end{para}

\begin{lem}
\label{lem:split exact seq}
Let $\cB$ be an additive category with countable coproducts and 
$a$ an object in $\cB$. 
We denote the induced morphism from the identity morphisms $at^i\to  a$ for 
non-negative integers $i$ by $\nabla_a \colon a[t] \to a$. 
Then the sequence
\begin{equation}
\label{eq:split exact}
a[t]\onto{\id_{a[t]}-t} a[t] \onto{\nabla_a} a
\end{equation}
is a split exact sequence in $\cB$. 
\end{lem}

\begin{proof}
We write $i_a\colon a\to a[t]$ 
for an inclusion functor $a=at^0\to\bigoplus_{i\geq 0}at^i$ 
and we define $q_a\colon a[t]\to a[t]$ to be a morphism in $\cB$ 
by sending ${(x_k)}_k$ to $\displaystyle{{\left (-\sum_{i\geq k+1}x_i \right )}_k}$. 
Then we can easily check the equalities 
$q_a(\id_{a[t]}-t)=\id_{a[t]}$, $\nabla_ai_a=\id_a$, 
$\nabla_a(\id_{a[t]}-t)=0$, $q_ai_a=0$ and 
$i_a\nabla_a+(\id_{a[t]}-t)q_a=\id_{a[t]}$. 
Hence the sequence $\mathrm{(\ref{eq:split exact})}$ is a split exact sequence.
\end{proof}

The following theorem is proved in \cite[9.10 b]{Sch00}.

\begin{thm}[\bf Abstract Hilbert basis theorem]
\label{thm:Abst Hilb basis}
For any noetherian object $a$ 
in an essentially small 
abelian category $\cA$, 
$a[t]$ is also a noetherian object in $\End\Lex \cA$.
\qed
\end{thm}

\begin{df}[\bf Schlichting polynomial category]
\label{df:S-poly cat def} 
Let us assume that $\cA$ is 
an essentially small noetherian abelian category and 
we denote the full subcategory of 
noetherian objects 
in $\End \Lex\cA$ 
by $\cA[t]$ and 
call $\cA[t]$ the {\it noetherian polynomial category} 
over $\cA$. 
By virtue of 
Lemma~\ref{lem:noetherian} and 
Theorem~\ref{thm:Abst Hilb basis}, 
we acquire the assertion that $\cA[t]$ is a noetherian abelian category. 
\end{df}

\begin{rem}
\label{rem:other df of S-poly cat}
We can prove that an object $x$ in $\End \Lex\cA$ is 
in $\cA[t]$ if and only if 
there exists a deflation $a[t] \rdef x$ for some object 
$a$ in $\cA$.
\end{rem}

\begin{ex}
\label{ex:morphismofa[t]}
For any noetherian objects $a$, $b$ in $\cA$ and 
a morphism $f:a[t] \to b[t]$ in $\cA[t]$, 
there exists a positive integer $m$ such that $f(a)$ is in 
$\displaystyle{\bigoplus_{i=1}^m bt^i}$. 
Since the morphism $f$ is recovered by the restriction 
$a \to a[t] \onto{f} b[t]$, 
$f$ is determined by 
morphisms $c_i:a \to b$ ($0\leq i\leq m$) in $\cA$. 
We write $f$ by $\displaystyle{\sum_{i=1}^mc_it^i}$.
\end{ex}

\begin{ex}
\label{ex:S-poly} 
Let $A$ be a ring with unit. 
Then we have the category equivalence
$$\cM_{A[t]}\isoto (\cM_A)[t],\ \ M \mapsto (M,t).$$
More precisely, 
by Remark~\ref{rem:GMnotation} and \ref{para:GQ-emmbeding}, 
we have the equivalences of categories
$$\Mod(A[t])\isoto\End \Mod(A)\isoto \End\Lex\cM_A.$$ 
By considering the full subcategories of consisting of 
those noetherian objects, 
we get the desired result. 
\end{ex}

\subsection{Abstract Artin-Rees lemma}
\label{subsec:abstract Artin-Ress lemma}

In this subsection, 
we prove an abstract version of Artin-Rees lemma.

\begin{df}[\bf $t$-filtration]
\label{df:t-filt}
Let $\cE$ an exact category and $X=(x,t)$ an object 
in $\End \cE$.\\
$\mathrm{(1)}$ 
A decreasing filtration $\fx=\{X_n=(x_n,t_n)\}_{n\geq 0}$ of $X$ in $\End\cE$, 
$$X=X_0\linf X_1\linf X_2\linf \cdots \linf X_n \linf \cdots$$
is a {\it $t$-filtration} if 
$\im(t_n:x_n \to x_n)\subset x_{n+1}$ for any $n\geq 0$.\\
$\mathrm{(2)}$ 
A $t$-filtration $\fx=\{X_n=(x_n,t_n)\}_{n\geq 0}$ is {\it stable} 
if there exists an integer $n_0\geq 0$ such that 
$\im(t_n:x_n \to x_n)\subset x_{n+1}$ for any $n\geq 0$.
\end{df}

\begin{df}[\bf Blow up]
\label{df:blow up}
Let $\cB$ an abelian category, 
$X=(x,t)$ an object in $\End \cB$ and 
$\fx=\{X_n=(x_n,t_n)\}_{n\geq 0}$ a $t$-filtration of $X$.\\
$\mathrm{(1)}$ 
We define an object $\Bl_{\fx}X$ in $\End\Lex\cB$ as follows. 
For any $n$, $t_n:x_n\to x_n$ induces a morphisms 
$t_n:x_n \to x_{n+1}$ and 
$\displaystyle{\bigoplus_{n\geq 0}t_n:\bigoplus_{n\geq 0}x_n 
\to \bigoplus_{n\geq 0} x_n}$. 
We put $\displaystyle{\Bl_{\fx}X:=\left (\bigoplus_{n\geq 0}x_n,\bigoplus_{n\geq 0}t_n \right )}$. 
We call $\Bl_{\fx}X$ a {\it blow up object of $X$ along $\fx$}.\\
$\mathrm{(2)}$ 
For each $n$, $\id_{x_n}:x_n \to x_n$ and 
the morphisms $t_{n+p-1}t_{n+p-2}\cdots t_{n+1}t_n:x_n \to x_{n+p}$ 
for $p>1$ induce a morphism 
$$\eta_{\fx}^n:X_n\to \Bl_{\fx}X$$
in $\End\Lex\cB$.
\end{df}

\begin{lem}
\label{lem:char of stability}
Let $\cA$ be a noetherian abelian category, 
$X=(x,t)$ an object in $\cA[t]$ and 
$\fx=\{X_n=(x_n,t_n)\}_{n\geq 0}$ a $t$-filtration in $\cA[t]$. 
Then the following conditions are equivalent:\\
$\mathrm{(1)}$ 
$\fX$ is stable.\\
$\mathrm{(2)}$ 
There exists an integer $m\geq 0$ such that the canonical morphism 
induced by $\eta_{\fx}^k$ {\rm (}$0\leq k\leq m${\rm )}, 
$\displaystyle{\bigoplus_{k=0}^mX_k \to \Bl_{\fx}X}$ is an epimorphism.\\
$\mathrm{(3)}$ 
$\Bl_{\fx}X$ is an object in $\cA[t]$, 
namely a noetherian object in $\End\Lex\cA$.
\end{lem}

\begin{proof}
We assume that there exists an integer $m\geq 0$ such that 
$\im(t_n\colon x_n \to x_n)=x_{n+1}$ for any $n\geq m$. 
Then obviously the canonical morphism 
$\displaystyle{\bigoplus_{k=0}^mX_k \to \Bl_{\fx}X}$ is an epimorphism. 

\sn
Next assume the condition $\mathrm{(2)}$. 
Since $\Bl_{\fx}X$ is a quotient of finite direct sum of noetherian objects 
in $\End\Lex\cA$, 
$\Bl_{\fx}X$ is noetherian by Lemma~\ref{lem:noetherian}. 

\sn
Finally we assume that $\Bl_{\fx}X$ is noetherian. 
We put $\displaystyle{z_m\colon =
\im\left (\bigoplus_{k=0}^mX_k \to \Bl_{\fx}X \right )}$. 
Then the sequence $z_0 \rinf z_1\rinf z_2\rinf \cdots $ is stational. 
Say $z_{n_0}=z_{n_0+1}=\cdots$. 
Then for any $n\geq n_0$, we have 
$$x_{n+1}\subset z_{n+1}\cap x_{n+1}=z_{n_0}\cap x_{n+1} 
\subset \sum_{i=0}^{n_0} \im (t_nt_{n-1}\cdots t_i\colon x_i \to x_{n+1})\subset 
\im (t_n\colon x_n \to x_{n+1}).$$
Hence $\fx$ is stable.
\end{proof}

\begin{cor}[\bf Abstract Artin-Rees lemma]
\label{cor:abst AR}
Let $(x,t_x)$  be an object in $\cA[t]$ and 
$(y,t_y)$ a subobject of $(x,t_x)$. 
Then there exist an integer $n_0\geq 0$ such that 
$$\im(t^n_x\colon x\to x)\cap y=\im(t_x^{n-n_0}\colon (\im(t^{n_0}:x\to x)\cap y) \to y)$$
for any $n\geq n_0$ in $\Lex\cA$.
\end{cor}

\begin{proof}
Consider the $t$-stable filtration 
$\fx=\{\im(t_x^n\colon x \to x),t_x \}_{n\geq 0}$ of 
$(x,t_x)$ and the induced $t$-filtration 
$\fy=\{\im(t_x^n\colon x \to x)\cap y,t_y \}_{n\geq 0}$
of $(y,t_y)$. 
Then $\Bl_{\fy}(y,t_y)$ is a subobject of $\Bl_{\fx}(x,t_x)$. 
Since $\Bl_{\fx}(x,t_x)$ is noetherian by Lemma~\ref{lem:char of stability}, 
$\Bl_{\fy}(y,t_y)$ is also noetherian and 
by Lemma~\ref{lem:char of stability} again, 
we learn that $\fy$ is stable. 
Hence we obtain the result.
\end{proof}

\section{Non-commutative motive theory over relative exact categories}
\label{sec:non-comm}

In this section, we will review the notions 
of additive and localizing theories over relative exact categories. 
Moreover we introduce a notion of nilpotent invariance.

\subsection{Relative exact categories}
\label{subsec:Rel Ex cat}

In this subsection, we recall jargons of relative exact categories 
from \cite{Moc13} and \cite{HM13}.

\begin{para}[\bf Relative exact categories]
\label{para:ext axiom}
$\mathrm{(1)}$ 
A {\it relative exact category} $\bE=(\cE,w)$ is a pair of 
an exact category $\cE$ 
with a specific zero object $0$ 
and a class of morphisms $w$ in $\cE$ 
which satisfies the following two axioms.\\
{\bf (Identity axiom).} 
For any object $x$ in $\cE$, 
the identity morphism $\id_x$ is in $w$.\\
{\bf (Composition closed axiom).} 
For any composable morphisms $\bullet \onto{a} \bullet \onto{b} \bullet$ 
in $\cE$, if $a$ and $b$ are in $w$, then $ba$ is also in $w$.\\
$\mathrm{(2)}$ 
A {\it relative exact functor} between relative exact categories 
$f:\bE=(\cE,w) \to (\cF,v)$ is 
an exact functor $f:\cE \to \cF$ such that 
$f(w)\subset v$ and $f(0)=0$. 
We denote the category of relative exact categories and relative exact functors 
by $\RelEx$.\\
$\mathrm{(3)}$ 
We write $\cE^w$ for the full subcategory of $\cE$ 
consisting of those object $x$ such that the canonical morphism 
$0 \to x$ is in $w$. 
We consider the following axioms.\\
{\bf (Strict axiom).} 
$\cE^w$ is an exact category such that 
the inclusion functor $\cE^w \rinc \cE$ is exact 
and reflects exactness.\\
{\bf (Very strict axiom).} 
$\bE$ satisfies the strict axiom and 
the inclusion functor $\cE^w \rinc \cE$ induces a fully faithful 
functor $\calD^b(\cE^w) \rinc \calD(\cE)$ on the bounded derived categories.\\
We denote the category of strict (resp. very strict) relative exact categories 
by $\RelEx_{\operatorname{strict}}$ 
(resp. $\RelEx_{\operatorname{vs}}$).\\
$\mathrm{(4)}$ 
A {\it relative natural equivalence} $\theta:f \to f'$ 
between relative exact functors $f$, $f':\bE=(\cE,w) \to \bE'=(\cE',w')$ 
is a natural transformation $\theta:f \to f'$ such that $\theta(x)$ is in $w'$ 
for any object $x$ in $\cE$. 
Relative exact functors $f$, $f':\bE \to \bE'$ are {\it weakly homotopic} 
if there is a zig-zag sequence of ralative natural equivalences 
connecting $f$ to $f'$. 
A relative exact functor $f:\bE \to \bE'$ is a {\it homotopy equivalence} 
if there is a relative exact functor $g:\bE' \to \bE$ 
such that $gf$ and $fg$ are 
weakly homotopic to identity functors respectively.\\
$\mathrm{(5)}$ 
A functor $F$ from a full subcategory $\calR$ of $\RelEx$ to 
a category $\cC$ is {\it categorical homotopy invariant} 
if for any relative exact functors $f$, $f':\bE \to \bE'$ in $\calR$ 
such that $f$ and $f'$ are weakly homotopic, 
we have the equality $F(f)=F(f')$. 
\end{para}

\begin{para}[\bf Derived category]
\label{para:derivd cat}
We define the {\it derived categories} of 
a strict relative exact category $\bE=(\cE,w)$ by the following formula
$$\calD_{\#}(\bE):=\coker(\calD_{\#}(\cE^w) \to \calD_{\#}(\cE))$$
where $\# =b$, $\pm$ or nothing. 
Namely $\calD_{\#}(\bE)$ is a Verdier quotient of $\calD_{\#}(\cE)$ 
by the thick subcategory of $\calD_{\#}(\cE)$ 
spanned by the complexes in $\Ch_{\#}(\cE^w)$.
\end{para}

\begin{para}[\bf Quasi-weak equivalences]
\label{para:quasi-weak equiv}
Let 
$P_{\#}:\Ch_{\#}(\cE) \to \calD_{\#}(\bE)$ 
be the canonical quotient functor. 
We denote the pull-back of the class of all isomorphisms in 
$\calD_{\#}(\bE)$ 
by $qw_{\#}$ or simply $qw$. 
We call a morphism in $qw$ a {\it quasi-weak equivalence}. 
We write $\Ch_{\#}(\bE)$ for a pair $(\Ch_{\#}(\cE),qw)$. 
We can prove that 
$\Ch_{\#}(\bE)$ is a complicial biWaldhausen 
category in the sense of \cite[1.2.11]{TT90}. 
In particular, it is a relative exact category. 
The functor $P_{\#}$ induces an equivalence of triangulated categories 
$\cT(\Ch_{\#}(\cE),qw)\isoto\calD_{\#}(\bE)$ 
where the category $\cT(\Ch_{\#}(\cE),qw)$ 
is the triangulated category associated with the category $(\Ch_{\#}(\cE),qw)$ 
(See \cite[3.2.17]{Sch11}). 
If $w$ is the class of all isomorphisms in $\cE$, 
then $qw$ is just the class of all quasi-isomorphisms in $\Ch_{\#}(\cE)$ 
and we denote it by $\qis$. 
\end{para}

\begin{para}[\bf Consistent axiom]
\label{para:consis axiom}
Let $\bE=(\cE,w)$ be a strict relative exact category. 
There exists the canonical functor $\iota^{\cE}_{\#}:\cE \to \Ch_{\#}(\cE)$ 
where $\iota^{\cE}_{\#}(x)^k$ is $x$ if $k=0$ and $0$ if $k\neq 0$. 
We say that $w$ (or $\bE$) satisfies 
the {\it consistent axiom} 
if $\iota^{\cE}_b(w)\subset qw$. 
We denote the full subcategory of 
consistent relative exact categories 
(resp. very strict consistent relative exact categories) 
in $\RelEx$ by 
$\RelEx_{\operatorname{consist}}$ (resp. $\RelEx_{\operatorname{vs,\ consist}}$). 
\end{para}

\begin{ex}
\label{ex:semi devices}
(\cf \cite{Moc13}).\ 
$\mathrm{(1)}$ 
A pair $(\cE,i_{\cE})$ of 
an exact category $\cE$ with the class of all isomorphisms $i_{\cE}$ 
is a very strict consistent relative exact category. 
We regard the category of essentially small exact categories $\ExCat$ 
as the full subcategory of $\RelEx_{\operatorname{vs,\ consist}}$ 
by the fully faithful functor 
$\ExCat \to \RelEx_{\operatorname{vs,\ consist}}$ 
which sends an exact category $\cE$ to a relative exact category 
$(\cE,i_{\cE})$. 
For simplicity, we sometimes write $\cE$ for $(\cE,i_{\cE})$.\\
$\mathrm{(2)}$ 
In particular we denote the trivial exact category by $0$ and 
we also write $(0,i_0)$ for $0$. 
$0$ is the zero objects in the category of consistent relative exact categories.\\
$\mathrm{(3)}$ 
A complicial exact category with weak equivalences in the sense of 
\cite[3.2.9]{Sch11} is a consistent relative exact category. 
In particular for any relative exact category $\bE$, 
$\Ch_{\#}(\bE)$ is a very strict consistent relative exact category.
\end{ex}

\begin{para}[\bf Derived equivalence]
\label{para:derived equivalence}
An exact functor $f:\bE \to \bF$ is a {\it derived equivalence} 
if $f$ induces an equivalence of triangulated categories on 
the bounded derived categories $\calD_b(\bE)\isoto\calD_b(\bF)$.
\end{para}

We give an example of derived equivalence exact functor 
by the proof of Corollary~3 of resolution theorem in \cite{Qui73} 
and \cite[3.2.8]{Sch11}:

\begin{para}[\bf Homology theory and acyclic objects]
\label{para:homology theory and acyclic objects}
$\mathrm{(1)}$ 
A {\it homology theory} on an exact category $\cE$ to an abelian category 
$\cB$ is an exact connected sequence of functors 
$T=\{T_n\}_{n\geq 1}$ from $\cE$ to $\cB$. 
Namely for any conflation $x\rinf y \rdef z$ in $\cE$, 
we have a long exact sequence 
$$\cdot \to T_2z \to T_1x \to T_1y\to T_1z.$$
$\mathrm{(2)}$ 
Let $T=\{T_n\}_{n\geq 1}$ be a homology theory on an exact category $\cE$. 
An object $x$ is {\it $T$-acyclic} if $T_nx=0$ for all $n\geq 1$.
\end{para}

\begin{lem}
\label{lem:resol thm}
Let $\cE$ be an exact category and $T$ a homology theory on $\cE$ and 
$\cE_{\text{\rm $T$-acy}}$ 
the full subcategory of $T$-acyclic objects in $\cE$. 
Assume for each $x$ in $\cE$ that there exists 
a deflation $y\rdef x$ with $y$ in $\cE_{\text{\rm $T$-acy}}$, and 
that $T_nx$ is trivial for $n$ sufficiently large. 
Then the inclusion functor $\cE_{\text{\rm $T$-acy}}\rinc \cE$ is 
a derived equivalence.
\qed
\end{lem}

\begin{para}[\bf Non-connective $K$-theory for (consistent) relative exact categories]
\label{df:non-connective K-theory}
(\cf \cite{Moc13}).\ 
For a consistent relative exact category 
$\bE=(\cE,w)$, we define the non-connective $K$-theory $\bbK(\bE)$ by the formula
$\bbK(\bE)=\bbK^S(\Ch_b(\bE))$ where $\bbK^S$ means the non-connective $K$-theory 
defined and studied by Schlichting in \cite{Sch06} or \cite{Sch11}. 
If either $w$ is the class of all isomorphisms or 
$\bE$ is a complicial exact 
category with weak equivalences in the sense of \cite{Sch11}, then 
the canonical morphism $\bE\to\Ch_b\bE$ induces an equivalence of 
spectra $\bbK^S(\bE)\isoto\bbK(\bE)$. 
The operation $\bbK$ becomes a functor 
from the category of essentially small consistent relative exact categories 
to the stable category of spectra.
\end{para}

\subsection{Additive theory}
\label{subsec:add theory}

\begin{para}
\label{para:Rel ex seq} 
Let $\bE=(\cE,w)$ be a relative exact category. 
We denote the exact category of admissible short exact sequences in $\cE$ 
by $E(\cE)$. 
There exist three exact functors $s$, $t$ and $q$ from 
$E(\cE) \to \cE$ which send 
an admissible exact sequence $x\rinf y \rdef z$ to 
$x$, $y$ and $z$ respectively. 
We write $w_{E(\bE)}$ for the class of morphisms 
$s^{-1}(w)\cap t^{-1}(w)\cap q^{-1}(w)$ and 
put $E(\bE):=(E(\cE),w_{E(\bE)})$. 
We can easily prove that $E(\bE)$ is a relative exact category 
and the functors $s$, $t$ and $q$ are relative exact functors 
from $E(\bE)$ to $\bE$. 
Moreover 
we can easily prove that 
if $\bE$ is consistent, then 
$E(\bE)$ is also consistent.
\end{para}

\begin{df}[\bf Additive theory]
\label{df:additive theory}
$\mathrm{(1)}$ 
A full subcategory $\calR$ of $\RelEx$ 
is {\it closed under extensions} if 
$\calR$ contains the trivial relative exact category $0$ and 
if for any $\bE$ in $\calR$, $E(\bE)$ is also in $\calR$.\\
$\mathrm{(2)}$ 
Let $\fA$ 
be a functor from a full subcategory $\calR$ of $\RelEx$ 
closed under extensions to an additive category $\cB$. 
We say that $\fA$ is an {\it additive theory} if 
for any relative exact category $\bE$ in $\calR$, 
the following projection is an isomorphism 
$$
\displaystyle{\begin{pmatrix}\fA(s)\\ \fA(q)\end{pmatrix}:
\fA(E(\bE)) \to \fA(\bE)\oplus \fA(\bE)}.
$$
\end{df}

By the proof of Crollary~2 of the additivity theorem in \cite{Qui73}, 
we get the additivity for characteristic filtration: 

\begin{para}[\bf Characteristic filtration]
\label{para:characteristic filtration}
A {\it characteristic filtration} of 
a functor $f$ between exact categories $\cE' \to \cE$ 
is a finite sequence $0=f_0 \to f_1 \to \cdots f_n=f$ of 
natural transformations between exact functors from $\cE'$ to $\cE$ 
such that $f_{p-1}(x) \to f_p(x)$ is an inflation in $\cE$ 
for every $x$ in $\cE'$ and $1\leq p\leq n$, 
and induced quotient functors $f_p/f_{p-1}$ are exact 
for $1\leq p\leq n$. 
\end{para}

\begin{lem}[\bf Additivity for characteristic filtration]
\label{lem:characteristic filtration}
Let $\fA:\ExCat \to \cB$ be a categorical homotopy invariant additive theory and 
$f:\cE \to \cE'$ be an exact functor between 
exact categories equiped with a characteristic filtration 
$0=f_0\subset \cdots \subset f_n=f$. Then 
$$\fA(f)=\sum_{p=1}^n\fA(f_p/f_{p-1}).$$
\qed
\end{lem}

\subsection{Localizing theory}
\label{subsec:loc theory}

\begin{df}[\bf Exact sequence]
\label{df:exact sequence}
$\mathrm{(1)}$ 
We say that a sequence of triangulated categories 
$\cT \onto{i} \cT' \onto{j} \cT''$ is {\it exact} if 
$i$ is fully faithful, the composition $ji$ is zero and 
the induced functor from $j$, $\cT'/\cT \to \cT''$ is {\it cofinal}. 
The last condition means that it is fully faithful and 
every object of $\cT''$ is 
a direct summand of an object of $\cT'/\cT$.\\
$\mathrm{(2)}$ 
A sequence $\bE \onto{u} \bF \onto{v} \bG$ 
of strict relative exact categories is 
{\it derived exact}
if the induced sequence of triangulated categories 
$\calD_b(\bE) \onto{\calD_b(u)} \calD_b(\bF) \onto{\calD_b(v)} \calD_b(\bG)$ is 
exact. 
We sometimes denote the sequence above by $(u,v)$. 
For a full subcategory $\calR$ of $\RelEx_{\operatorname{strict}}$, 
we let $E(\calR)$ denote the category of 
exact sequences in $\calR$. 
We define three functors $s^{\calR}$, $m^{\calR}$ and $q^{\calR}$ from 
$E(\calR)$ to $\calR$ which send an exact sequence 
$\bE \to \bF \to \bG$ to $\bE$, $\bF$ and $\bG$ respectively.
\end{df}

\begin{ex}[\bf Exact sequence of abelian categories]
\label{ex:exact seq of ab cat}
Let $\cS$ be a Serre subcategory of an abelian category $\cB$. 
Then the canonical sequence 
$$\cS \to \cB \to \cB/\cS$$
is derived exact 
if $\cS$ and $\cB$ satisfy the following condition $\mathrm{(\ast)}$:

\sn
$\mathrm{(\ast)}$ 
For any monomorphism $x\rinf y$ in $\cB$ with $x$ in $\cS$, 
there exists a morphism $y \to z$ with $z$ in $\cS$ such that 
the composition $x \to y \to z$ is a monomorphism. 
(See \cite[4.1]{Gro77} and \cite[1.15]{Kel99}).
\end{ex}

\begin{df}[\bf Localizing theory]
\label{df:Loc th}
$\mathrm{(1)}$ 
A {\it localizing theory} $(\frakL,\partial)$ 
from a full subcategory $\calR$ of $\RelEx_{\operatorname{strict}}$ to 
a triangulated category $(\cT,\Sigma)$ is a pair 
of functor $\frakL:\calR \to \cT$ and a natural transformation 
$\partial:\frakL q \to \Sigma \frakL s$ between functors 
$E(\calR) \substack{\overset{s}{\to}\\ \underset{q}{\to}} 
\calR \onto{\frakL} \cT$ 
which sends a derived exact sequnece $\bE \onto{i} \bF \onto{j}\bG$ in $\calR$ 
to a distingushed triangle 
$\frakL(\bE)\onto{\frakL(i)}\frakL(\bF)\onto{\frakL(j)}\frakL(\bG)
\onto{\partial_{(i,j)}}\Sigma \frakL(\bE)$ in $\cT$.\\
$\mathrm{(2)}$ 
A localizing theory $(\frakL,\partial)$ is {\it fine} if $\frakL$ is 
a categorical homotopy invariant functor 
and $\frakL$ commutes with filtered colimits.
\end{df}

\begin{rem}
\label{rem:localizing theory}
$\mathrm{(1)}$ 
The non-connective $K$-theory on $\RelEx_{\operatorname{consist}}$ 
studied in \cite{Sch11}, \cite{Moc13} 
is a fine localization theory.\\
$\mathrm{(2)}$ 
(\cf \cite[7.9]{Moc13}). 
Let $\frakL$ be a localization theory on a full subcategory $\calR$ 
of $\RelEx_{\operatorname{strict}}$. 
Then\\
$\mathrm{(i)}$ 
$\frakL$ is a derived invariant functor. 
Namely 
if a morphism $\bE \to \bF$ in $\calR$ 
is a derived equivalence, 
then the induced morphism 
$\frakL(\bE)\to \frakL(\bF)$ is an isomorphism.\\ 
$\mathrm{(ii)}$ 
If further we assume that $\calR$ is closed under extensions and 
if $\frakL$ is categorical homotopy invariant, 
then we can easily prove that $\frakL$ is an additive theory. 
\end{rem}

\subsection{Nilpotent invariance}
\label{subsec:nilp inv th}

In this subsection, we define the notion about nilpotent invariant functors.

\begin{df}[\bf Serre radical]
\label{df:Serre radical}
Let $\cB$ be an abelian category and $\cF$ a full subcategory of $\cB$. 
We write ${}^S\!\!\!\sqrt{\cF}$ for intersection of all Serre subcategories 
which contain $\cF$ and call it the {\it Serre radical of $\cF$}. 
\end{df}

For noetherian abelian categories, 
we give a characterization of Serre radicals of full subcategories.

\begin{df}[\bf Admissible subquotient]
\label{df:subquotient}
Let $\cE$ be an exact category and $a$ and $b$ objects in $\cE$. 
We say that $a$ is an {\it admissible subquotient} of $b$ 
if there exists a filtration of inflations 
$$b=b_0\linf b_1\linf b_2$$
such that $b_1/b_2\isoto a$.
\end{df}

\begin{prop}
\label{prop:Serre radical}
{\rm (\cf \cite[3.1]{Her97}, \cite[2.2]{Gar09}).} 
Let $\cB$ be a noetherian abelian category, 
$\cF$ a full subcategory of $\cB$ and 
$x$ an object in $\cB$. 
Then $x$ is in ${}^S\!\!\!\sqrt{\cF}$ if and only if there exists 
a finite filtration of admissible subobjects 
$$x=x_0\linf x_1\linf x_2\linf x_3\linf \cdots \linf x_n=0$$
such that for every $i<n$, $x_i/x_{i+1}$ is an admissible subquotinet 
of an object $\cF$.
\qed
\end{prop}

\begin{df}[\bf Nilpotent invariance]
\label{df:nipotent invariance}
Let $\calR$ be a full subcategory of $\RelEx$ which contains 
$\AbCat$ the category of essentially small abelian categories. 
A functor $\fN\colon\calR \to \cC$ is 
{\it nilpotent invariant} if 
for any noetherian abelian category 
$\cB$ and any full subcategory $\cF$ such that 
${}^S\!\!\!\sqrt{\cF}=\cB$ and $\cF$ is closed under 
finite direct sums, sub- and quotient objects, 
the inclusion functor $\cF \rinc \cB$ induces an isomorphism 
$\fN(\cF)\isoto\fN(\cB)$ in $\cC$. 
\end{df}

\begin{ex}
\label{ex:nilpotent invariance}
The connective and the non-connective $K$-theory are nilpotent 
invariant by d\'evissage theorem in \cite{Qui73} and 
Theorem~7 in \cite{Sch06}. 
\end{ex}

\section{Graded categories}
\label{sec:graded cat}

In this section, we will introduce the notion of 
(noetherian) graded categories 
over categories and calculate a fine localizing theory of 
noetherian graded categories 
over noetherian abelian categories.

\subsection{Fundamental properties of graded categories}
\label{subsec:grad cat}

As in the 
results \cite{Ser55}, \cite{AZ94}, \cite{Pol05} and \cite{GP08}, 
the category of finitely generated graded objects understudies 
the category of coherent sheaves over projective spaces. 
We define the notion of graded categories over categories and study 
the fundamental properties. 
It is an abstract version of graded modules. 
See for the motivational Example~\ref{ex:gradedcategory}. 

\begin{para}
\label{para:cat<n>}
For a positive integer $n$, we define the category 
$\langle n\rangle$ as follows. 
The class of objects of $\langle n\rangle$ is just the set of all natural numbers 
$\bbN$. 
The class of morphisms of $\langle n\rangle$ is generated by 
morphisms $\psi^{i}_m\colon m \to m+1$ 
for any $m$ in $\bbN$ and $1\leq i\leq n$ 
which subject to the equalities 
$\psi^i_{m+1}\psi^j_m=\psi^j_{m+1}\psi^i_m$ 
for each $m$ in $\bbN$ and $1\leq i,\ j \leq n$. 
\end{para}

\begin{df}[\bf Graded categories]
\label{df:graded cat}
For any positive integer $n$ and any category $\cC$, 
we put $\cC_{\gr}[n]:=\HOM(\langle n\rangle,\cC)$ and call it 
the {\it category of {\rm (}$n$-{\rm )}graded category over $\cC$}. 
For any object $x$ and 
any morphism $f\colon x \to y$ in $\cC_{\gr}[n]$, 
we denote $x(m)$, $x(\psi_m^i)$ and $f(m)$ 
by $x_m$, $\psi_m^{i,x}$ or shortly 
$\psi_m^i$ and $f_m$ respectively. 
\end{df}

\begin{rem}
\label{rem:limit in graded cat}
We can calculate a (co)limit in $\cC_{\gr}[n]$ by 
term-wise (co)limit in $\cC$. 
In particular, if 
$\cC$ is additive (resp. abelian) 
then $\cC_{\gr}[n]$ is also additive (resp. abelian). 
Moreover if 
$\cC$ is a category with cofibration 
(resp. an exact category), 
then $\cC_{\gr}[n]$ naturally becomes 
a category with cofibration 
(resp. an exact category). 
Here a sequence 
$x \to y \to z$ is a conflation 
(resp. a morphism $x \to y$ is a cofibration) 
if it is term-wisely in $\cC$. 
Moreover if $w$ is a class of morphisms in $\cC$ 
which satisfies the axioms of Waldhausen categories 
(and its dual), then 
 the class of all morphisms $lw$ in $\cC_{\gr}[n]$ 
consisting of 
those morphisms $f$ such that $f_m$ is in $w$ for 
all natural number $m$ also satisfies the axioms 
of Waldhausen categories (and its dual).
\end{rem}

We can prove the following lemma and corollary: 

\begin{lem}
\label{lem:Hom hereditary}
Let $\cC$, $\calD$ and $\cI$ be categories and $f:\cC \to \calD$ a functor. 
If $f$ is faithful {\rm (}resp. fully faithful{\rm )}, 
then $\HOM(\cI,f):\HOM(\cI,\cC) \to \HOM(\cI,\calD)$ is faithful 
{\rm (}resp. fully faithful{\rm )}.
\qed
\end{lem}

\begin{cor}
\label{cor:Hom hereditary}
Let $f:\cC \to \calD$ be a functor between categories 
and $n$ a positive integer. 
If $f$ is faithful {\rm (}resp. fully faithful{\rm )}, 
then the induced functor $f_{\gr}[n]:\cC_{\gr}[n] \to \calD_{\gr}[n]$ 
is faithful 
{\rm (}resp. fully faithful{\rm )}.
\qed
\end{cor}

\begin{para}
\label{para:cC'df}
For an exact category $\cE$ and a positive integer $n$, 
we denote the full subcategory of all noetherian objects 
in $\cE_{\gr}[n]$ by $\cE_{\gr}'[n]$. 
In particular if $\cE$ is an abelian category then 
$\cE'_{\gr}[n]$ is a noetherian abelian category 
by Lemma~\ref{lem:noetherian}. 
In this case, we call $\cE'_{\gr}[n]$ the 
{\it noetherian ($n$-)graded category over $\cE$}. 
\end{para}

\begin{df}[\bf Degree shift]
\label{nt:degreeshiftofobjects}
Let $\cC$ be a category with a specific zero object $0$ 
and $k$ an integer. 
We define the functor 
$(k):\cC_{\gr}[n] \to \cC_{\gr}[n]$, $x \mapsto x(k)$. 
For any object $x$ and 
any morphism $f:x \to y$ in $\cC_{\gr}[n]$, 
we define an object $x(k)$ and 
a morphism $f(k):x(k) \to y(k)$ in $\cC_{\gr}[n]$ as follows. 
We put 
$$
x(k)_m=
\begin{cases}
x_{m+k} & \text{if $m\geq -k$}\\
0 & \text{if $m < -k$}
\end{cases},\ \ 
\psi^{i,x(k)}_m:=
\begin{cases}
\psi^{i,x}_{m+k} & \text{if $m\geq -k$}\\
0 & \text{if $m < -k$}
\end{cases}
\ \ \text{and}\ \ 
f(k)_m:=
\begin{cases}
f_{m+k} & \text{if $m\geq -k$}\\
0 & \text{if $m < -k$}
\end{cases}
.$$
For any object $x$ in $\cC_{\gr}[n]$ and any positive integer $k$, 
we have the canonical morphism $\psi_x^{i,k}(=\psi^i):x(-k) \to x(-k+1)$ 
defined by $\psi_{m-k}^i:{x(-k)}_m=x_{m-k} \to {x(-k+1)}_m=x_{m-k+1}$ 
for each $m$ in $\bbN$. 
We consider a pair $(\cC_{\gr}[n],(-1))$ as an object 
in $\End \Cat$.
\end{df}

\begin{rem}
\label{rem:degree shift}
If $\cE$ is an exact category, then 
for any intger $k$, the functor 
$(k):\cE_{\gr}[n] \to \cE_{\gr}[n]$ is exact. 
Moreover this functor induce the exact functor 
$(k):\cE_{\gr}'[n] \to \cE_{\gr}'[n]$.
\end{rem}

\begin{df}
\label{nt:compoftranslation}
For any natural numbers $m$ and $k$, 
any object $x$ in $\cC_{\gr}[n]$ 
and any multi index $\ii=(i_1,\cdots,i_n)\in\bbN^n$, 
we define the morphism 
$\psi^{\ii,k}_x(=\psi^{\ii}):x(-(\sum_{j=1}^n i_j+k)) 
\to x(-k)$ by 
$$\psi^{\ii}={(\psi^n)}^{i_n}{(\psi^{n-1})}^{i_{n-1}}\cdots
{(\psi^2)}^{i_2}{(\psi^1)}^{i_1}.$$
\end{df}

\begin{df}[\bf Free graded object]
\label{nt:free graded object}
Let $\cC$ be an additive category 
and $n$ a positive integer. 
We define the functor 
$\cF_{\cC}[n](=\cF[n]):\cC \to \cC_{\gr}[n]$ 
in the following way. 
For any object $x$ in $\cC$, 
we define the object 
$\cF[n](x)=x[\{\psi^i\}_{1\leq i\leq m}]$ 
in $\cC_{\gr}[n]$ as follows. 
We put 
$$\cF[n](x)_m:=\underset{
\substack{\ii=(i_1,\cdots,i_n)\in\bbN^n \\ 
\sum^n_{j=1}i_j=m}}{\bigoplus} 
x_{\ii}$$ 
where $x_{\ii}$ is a copy of $x$. 
$x_{\ii}$ ($\displaystyle{\sum_{j=1}^n i_j=m}$) components of the morphisms 
$\psi_m^{k,\cF[n](x)}:\cF[n](x)_m \to \cF[n](x)_{m+1}$  
defined by  
$\id:x_{\ii} \to x_{\ii+\fe_k}$ 
where $\fe_k$ is the $k$-th unit vector. 
\end{df}

\begin{para}
\label{para:can mor} 
Let $\cC$ be an additive category and $k$ a natural number. 
For any object $x$ in $\cC_{\gr}[n]$, 
we have the canonical morphism 
$\cF[n](x_k)(-k) \to x$ which is defined as follows.
For any $m\geq k$ and any $\ii=(i_1,\cdots,i_n)\in\bbN^n$ 
such that $\displaystyle{\sum_{j=1}^n i_j=m-k}$, 
on the $x_{\ii}$ component of $\cF[n](x_k)(-k)_m$, 
the morphism is defined by 
$\psi_m^{\ii}:x_{\ii} \to x_m$. 
\end{para}

\begin{rem}
\label{rem:adjointnessofF[n]}
Let $\cC$ be an additive category. 
Then the functor $\cF[n]:\cC \to \cC_{\gr}[n]$ is 
the left adjoint functor of the functor 
$\cC_{\gr}[n] \to \cC$, $y \mapsto y_0$. 
Namely for any object $x$ in $\cC$ 
and any object $y$ in $\cC_{\gr}[n]$, we have 
a functorial isomorphism 
$\Hom_{\cC}(x,y_0)\isoto\Hom_{\cC_{\gr}[n]}(\cF[n](x),y)$, 
which sends $f$ to 
$(\cF[n](x)\onto{\cF[n](f)}\cF[n](y_0) \to y)$.
\end{rem}

\begin{ex}
\label{ex:mor between freeobj} 
For any objects $x$ and $y$ in an additive category $\cC$, 
any positive integer $k$,  
and any family of morphisms 
$\{c_{\ii}\}_{\ii=(i_1,\cdot,i_n)\in\bbN^n,\ \sum i_j=k}$ 
from $x$ to $y$, 
we define the morphism $\sum c_{\ii}\psi^{\ii}:\cF[n](x)(-k) \to \cF[n](y)$ 
by $c_{\ii}:x_{\jj} \to x_{\jj+\ii}$ on its $x_{\jj}$ component to 
$x_{\jj+\ii}$ component. 
\end{ex}

\begin{lem}
\label{lem:fundamental facts about}
Let $\cA$ be a noetherian abelian category and 
$n$ a positive integer. 
Then\\
$\mathrm{(1)}$ 
For any object $x$ in $\cA$, 
$\cF[n](x)$ is a noetherian object in $\cA_{\gr}[n]$. 
In particular, 
we have the exact functor 
$$\cF_{\cA}[n]:\cA \to \cA'_{\gr}[n].$$
$\mathrm{(2)}$ 
For any object $x$ in $\cA'_{\gr}[n]$, 
there exists a natural number $m$ such that the canonical morphism 
as in \ref{para:can mor} 
$$\displaystyle{\overset{m}{\underset{k=0}{\bigoplus}}
\cF[n](x_k)(-k) \to x}$$ 
is an epimorphism.
\end{lem}

\begin{proof}
$\mathrm{(1)}$ 
We define the functor 
$$\Gamma\colon \cA_{\gr}[n] \to \End^n\Lex\cA,\ x\mapsto 
\left (\bigoplus x_{\ii},\ \bigoplus\psi_m^1,\cdots,
\bigoplus \psi_m^n \right ) $$
where $\End^n$ means the $n$-times iteration of the functor $\End$. 
Since $\Lex \cA$ is Grothendieck abelian, 
the functor $\bigoplus$ is exact and 
therefore $\Gamma$ is an exact functor. 
Moreover for any morphism $f:x\to y$ in $\cA_{\gr}[n]$, 
the condition $\Gamma(f)=0$ obviously implies 
the condition $f=0$. 
Hence $\Gamma$ is faithful. 
We can easily check that for any object $x$ in $\cA$, 
we have the canonical isomorphism 
$\Gamma(\cF[n](x))\isoto x[t_1,\cdots,t_n]$ 
and 
$x[t_1,\cdots,t_n]$ is a noetherian object 
in $\End^n\Lex\cA$ by Theorem~\ref{thm:Abst Hilb basis}. 
Therefore $\cF[n](x)$ is noetherian 
in $\cA_{\gr}[n]$ 
by Lemma~\ref{lem:faithfulexact}. 

\sn
$\mathrm{(2)}$ 
We put $\displaystyle{z_l=\im\left (\bigoplus_{k=0}^l \cF[n](x_k)(-k) \to x\right )}$. 
Let us consider the ascending chain of subobjects in $x$
$$z_1\rinf z_2\rinf \cdots \rinf x.$$ 
Since $x$ is a noetherian object, 
there exists a natural number $m$ such that 
$z_m=z_{m+1}=\cdots$. 
We claim that the canonical morphism 
$$y:=\bigoplus_{k=0}^i\cF[n](x_k)(-k) \to x$$
is an epimorphism. 
If $k\leq m$, $y_k \to x_k$ is obviously an epimorphism. 
If $k>m$, then we have the equalities
$$\im(y_k \to x_k)={(z_m)}_k={(z_k)}_k=x_k.$$ 
Therefore we get the desired result.
\end{proof}

\begin{df}[\bf Finitely generated objects]
\label{nt:finitely generated obj}
Let $\cE$ be an exact category.\\
$\mathrm{(1)}$ 
An object $(x,u)$ in $\End \cE$ is {\it finitely generated} 
if there exists an object $y$ in $\cE$ 
and an epimorihsim $(y[t],t) \rdef (x,u)$ in $\End (\Lex \cE)$. 
Let us write $\End (\Lex\cE)_{\fingen}$ for 
the full subcategory of $\End(\Lex\cE)$ consisting of 
those finitely generated objects in $\End(\Lex\cE)$.\\
$\mathrm{(2)}$ 
An object $x$ in $\cE_{\gr}[n]$ is {\it finitely generated} 
if there exists a non-negative integer $n$ such that 
the canonical morphism 
$\displaystyle{\bigoplus_{k=0}^n\cF[n](x_k)(-k) \to x}$ 
as in Remark~\ref{rem:adjointnessofF[n]} 
is an epimorphism. 
We denote the full subcategory of $\cE_{\gr}[n]$ consisting 
of those finitely generated objects in $\cE_{\gr}[n]$ 
by $\cE_{\gr}[n]_{\fingen}$.
\end{df}

\begin{rem}
\label{rem:finitely generated} 
Let $f:\cB \to \cC$ be an exact functor from an exact category $\cB$ 
to an exact category $\cC$. 
Then\\
$\mathrm{(1)}$ 
For any object $x$ in $\cB$, we have the 
equality 
$f_{\gr}[n](\cF[n](x))=\cF[n](f(x))$.\\
$\mathrm{(2)}$ 
Therefore if $\cB$ is an abelian category, then 
$f$ induces an exact functor 
$f_{\gr}[n]_{\fingen}:\cB_{\gr}[n]_{\fingen} \to \cC_{\gr}[n]_{\fingen}$.\\
$\mathrm{(3)}$ 
Moreover if $\cB$ is an essentially small noetherian abelian category, 
then we have $\cB_{\gr}[n]_{\fingen}=\cB'_{\gr}[n]$ and 
$\End (\Lex \cB)_{\fingen}=\cB[t]$ by Lemma~\ref{lem:noetherian} $\mathrm{(1)}$, 
Remark~\ref{rem:other df of S-poly cat} and 
Lemma~\ref{lem:fundamental facts about}. 
\end{rem}

\begin{ex}
\label{ex:gradedcategory}
For a ring with unit $A$ and 
$\cE=\cM_A$, 
$\cE_{\gr}[n]_{\fingen}$ is just the category of 
finitely generated graded right $A[t_1,\cdots,t_n]$-modules 
$\cM_{A[t_1,\cdots,t_n],\gr}$.
\end{ex}

\begin{proof}
Any object $x$ in $\cM_{A[t_1,\cdots,t_n],\gr}$ 
is considered to be an object in $\cE_{\gr}[n]_{\fingen}$ 
in the following way. 
Let us define the functor $x'\colon\langle n\rangle \to \cE$ 
by $k \mapsto x_k$ and $(\psi^i:k\to k+1)\mapsto (t_i:x_k \to x_{k+1})$. 
The association $x \mapsto x'$ induces 
a category equivalence 
$\cM_{A[t_1,\cdots,t_n],\gr} \isoto \cE_{\gr}[n]_{\fingen}$.
\end{proof}

\begin{df}[\bf Canonical filtration] 
\label{df:canfilt}
For any object $x$ in $\cA'_{\gr}[n]$, 
we define the canonical filtration $F_{\bullet}x$ as follows. 
$F_{-1}x=0$ 
and for any $m\geq0$,
$$(F_mx)_k=
\begin{cases}
x_k & \text{if $k\leq m$}\\
\displaystyle{\sum_{\substack{\ii=(i_1,\cdots,i_n)
\in\bbN^n\\\Sigma i_j =k-m}} 
\im\psi_{m}^{\ii}} & \text{if $k> m$}
\end{cases}.$$
\end{df}

\begin{rem}
\label{rem:dim of x}
Since every object $x$ in $\cA'_{\gr}[n]$ 
is noetherian, there is the minimal integer $m$ such that 
$F_mx=F_{m+1}x=\cdots$. 
In this case, 
we can easily prove that $F_mx=x$. 
We call $m$ {\it degree} of $x$ and denote it by $\deg x$. 
\end{rem}

\subsection{Koszul homologies}
\label{subsec:Kos hom}

In this subsection, we define the Koszul homologies of 
objects in $\cA'_{\gr}[n]$ and as an application of 
the notion about Koszul homologies, 
we study a fine localizing theory of $\cA'_{\gr}[n]$. 
There is the ideology about {\it Koszul duality} for example \cite{Bei78} 
in this subsection 
behind the use of the Koszul homologies. 

\begin{df}[\bf Koszul complex]
\label{nt:Koszul complex}
Let $\cC$ be an additive category and $n$ a positive integer. 
For any object $x$ in $\cC_{\gr}[n]$, 
we define the {\it Koszul complex} $\Kos(x)$ associated with $x$ 
as follows. 
$\Kos(x)$ is a chain complex in $\cC_{\gr}[n]$ concentrated in 
degrees $0,\cdots,n$ whose component at degree $k$ is given by 
$\displaystyle{\Kos(x)_k\colon 
=\bigoplus_{\substack{\ii=(i_1,\cdots,i_n)\in[1]^n 
\\ \sum i_j=k}} 
x_{\ii}}$ 
where $[1]$ is the totally ordered set 
$\{0,1\}$ with the natural order and 
$x_{\ii}$ is a copy of $\displaystyle{x(-\sum_{j=1}^n i_j)}$ and 
whose boundary morphism $d^{\Kos(x)}_k\colon \Kos(x)_k \to \Kos(x)_{k-1}$ is 
defined by $\displaystyle{(-1)^{\sum_{t=j+1}^n i_j}\psi^j\colon x_{\ii} 
\to x_{\ii-\fe_j}}$ on its $x_{\ii}$ to $x_{\ii-\fe_j}$ component where 
$\fe_j$ is the $j$-th unit vector. 
The association $x \mapsto \Kos(x)$ defines the exact functor 
$$\Kos:\cC_{\gr}[n] \to \Ch(\cC_{\gr}[n]).$$
\end{df}

\begin{df}[\bf Koszul homologies]
\label{df:Koszul homologies}
Let $\cE$ be an idempotent complete exact category and $n$ a positive integer. 
We put $\cB:=\Lex \cE$. 
We define the family of functors $\{T_i:\cE_{\gr}[n] \to \cB_{\gr}[n]\}$ by 
$T_i(x):=\Homo_i(\Kos(x))$ for each $x$. 
$T_i(x)$ is said to be the {\it $i$-th Koszul homology} of $x$. 
Let us notice that for any conflation $x \rinf y \rdef z$ in 
$\cE_{\gr}[n]$, we have a long exact sequence 
$$\cdots \to T_{i+1}(z) \to T_i(x) \to T_i(y) \to T_i(z) 
\to T_{i-1}(x) \to \cdots.$$
\end{df}

\begin{df}[\bf Torsion free objects]
\label{df:Torson free object}
An object $x$ in $\cA'_{\gr}[n]$ is {\it torsion free} if 
$T_i(x)=0$ for any $i>0$. 
For each non-negative integer $m$, 
we denote the category of torsion free objects 
(of degree less than $m$) 
in $\cA'_{\gr}[n]$ by 
$\cA'_{\gr,\tf}[n]$ (resp. $\cA'_{\gr,\tf,m}[n]$). 
Since $\cA'_{\gr,\tf}[n]$, $\cA'_{\gr,\tf,m}[n]$ are 
closed under extensions in $\cA'_{\gr}[n]$, 
they become exact categories in the natural way. 
\end{df}

\begin{prop}
\label{prop:fun pro Kos hom}
For any objects $x$ in $\cA'_{\gr}[n]$ and $y$ in $\cA$, 
we have the following assertions.\\
$\mathrm{(1)}$ 
For any natural number $k$, 
$\cF[n](y)(-k)$ is torsion free.\\
$\mathrm{(2)}$ 
For any positive integer $s$, 
the assertion 
$T_0(x)_k=0$ for any $k \leq s$ implies $x_k=0$ for any $k \leq s$.\\
$\mathrm{(3)}$ 
We have the equality 
$$T_0(F_px)_k=
\begin{cases}
0 & \text{if $k>p$}\\
T_0(x)_k & \text{if $k\leq p$}
\end{cases}
.$$
$\mathrm{(4)}$ 
For any natural number $p$, 
there exists a canonical epimorphism 
$$\alpha^p\colon \cF[n](T_0(x)_p)(-p) \rdef F_px/F_{p-1}x.$$
$\mathrm{(5)}$ 
For any natural number $p$, 
$T_0(\alpha^p)$ is an isomorphism.\\
$\mathrm{(6)}$ 
If $T_1(x)$ is trivial, then $\alpha^p$ is an isomorphism.
\end{prop}

\begin{proof}
$\mathrm{(1)}$ Since the degree shift functor is exact, 
we have the equality $T_i(x(-k))=T_i(x)(-k)$ 
for any natural numbers 
$i$ and $k$. 
Therefore we shall just check that 
$\cF[n](y)$ is torsion free. 
If $\cA$ is the category of finitely generated free $\bbZ$-modules 
$\cP_{\bbZ}$ and $y=\bbZ$, 
then $\cF[n](y)$ is just the $n$-th polynomial 
ring over $\bbZ$, $\bbZ[t_1,\cdots,t_n]$ 
and $T_i(\cF[n](y))$ is the $i$-th homology group of 
the Koszul complex associated with the regular sequence 
$t_1,\cdots,t_n$. 
In this case, 
it is well-known that 
$T_i(\cF[n](y))=0$ for $i>0$. 
For general $\cA$ and $y$, 
there exists an exact functor 
$\cP_{\bbZ} \to \cA$ which sends $\bbZ$ to $y$ and 
which induces $\Ch({(\cP_{\bbZ})}'_{\gr}[n]) \to \Ch(\cA'_{\gr}[n])$ 
and $\Kos(\cF[n](\bbZ))$ goes to $\Kos(\cF[n](y))$ by 
this exact functor. 
Hence we obtain the equality $T_i(\cF[n](y))=0$ for any positive integer $i$. 

\sn
$\mathrm{(2)}$ 
First notice that we have the equalities 
$${T_0(x)}_k=
\begin{cases}
x_0 & \text{if $k=0$}\\
x_k/\im(\psi^1,\cdots,\psi^n) & \text{if $k>0$}
\end{cases}
.
$$
Therefore if 
${T_0(x)}_k=0$ for $k\leq s$, then we have 
$x_0=0$ and $x_k=\im(\psi^1,\cdots,\psi^n)$ for $k\leq s$. 
Hence inductively we notice that $x_k=0$ for $k\leq s$.

\sn
Assertion 
$\mathrm{(3)}$ follows from direct calculation.

\sn
$\mathrm{(4)}$ 
We have the equality 
$$
{(F_px/F_{p-1}x)}_k \isoto
\begin{cases}
0 & \text{if $k < p $}\\
x_p/\im(\psi^1,\cdots,\psi^n)={T_0(x)}_p & \text{if $k=p$}
\end{cases}
.
$$
Therefore by 
Remark~\ref{rem:adjointnessofF[n]}, 
we  have the canonical morphism
$$\alpha^p\colon 
\cF[n]({T_0(x)}_p)(-p) \to ((F_px/F_{p-1}x)(p))(-p)=F_px/F_{p-1}x.$$
One can check that the morphism is an epimorphism. 

\sn
$\mathrm{(5)}$ 
By $\mathrm{(1)}$, 
we have the equalities 
$$
{F_px/F_{p-1}x}\isoto{T_0(\cF[n]({T_0(x)}_p)(-p))}_k \isoto
\begin{cases}
0 & \text{if $k\neq p $}\\
x_p/\im(\psi^1,\cdots,\psi^n) & \text{if $k=p$}
\end{cases}
$$
and ${T_0(\alpha^p)}_p=\id$. 
Hence we get the assertion. 

\sn
$\mathrm{(6)}$ 
Let $K^p$ be the kernel of $\alpha^p$, 
we have short exact sequences 
$$K^p \rinf \cF[n]({T_0(x)}_p)(-p) \rdef F_px/F_{p-1}x,$$
$$F_{p-1}x \rinf F_p x \rdef F_px/F_{p-1}x.$$
We call the long exact sequences of Koszul homologies 
associated with short sequences above $\mathrm{(I)}$, 
$\mathrm{(II)}$ respectively. 
By $\mathrm{(I)}$ 
and assertions 
$\mathrm{(1)}$ and $\mathrm{(5)}$, 
we have the isomorphism 
$$T_1(F_px/F_{p-1}x)\isoto T_0(K^p).$$ 
We claim that the following assertion.

\begin{claim}
$T_1(F_px/F_{p-1}x)=0$ and $T_1(F_px)=0$.
\end{claim}

\sn
We prove the claim by descending induction of $p$. 
For sufficiently large $p$, 
we have $T_1(F_px)=T_1(x)$ 
and therefore it is trivial by the assumption. 
Then by $\mathrm{(II)}$ and $\mathrm{(3)}$, 
we have 
$$T_0(K^p)=T_1(F_px/F_{p-1}x)=0.$$ 
Therefore by $\mathrm{(2)}$, 
we have $K^p=0$. 
By $\mathrm{(I)}$ and $\mathrm{(1)}$, 
we have isomorphisms 
$$0=T_2(\cF[n]({T_0(x)}_p)(-p))
\isoto T_2(F_px/F_{p-1}x).$$ 
By $\mathrm{(II)}$, 
we get $T_1(F_{p-1}x)=0$. 
Hence we prove the claim and by $\mathrm{(2)}$, 
we get the desired result. 
\end{proof}

For an object $x$ in an additive category $\cB$, 
recall the definition 
of the polynomial object $x[t]$ in $\End\cB$ from \ref{para:polynomial category}. 
We regard $\ExCat$ 
the category of essentially small exact categories as 
the full subcategory or $\RelEx_{\operatorname{vs,\ consisit}}$. 
(See Example~\ref{ex:semi devices} $\mathrm{(1)}$).

\begin{thm}
\label{cor:canisomoftf}
$\mathrm{(1)}$ 
The inclusion functor $\cA'_{\gr,\tf}[n] \rinc \cA'_{\gr}[n]$ 
is a derived equivalence.\\
$\mathrm{(2)}$ 
Let $\fA:\ExCat \to \cC$ be a categorical homotopy invariant 
additive theory. 
Then for any natural number $m$, 
the exact functor $a:\cA'_{\gr,\tf,m}[n] \to \cA^{\times m+1}$ 
which is defined by sending an object $x$ in $\cA'_{\gr,tf,m}[n]$ 
to ${(T_0(x)_k)}_{0\leq k\leq m}$ in $\cA^{\times m+1}$ induces an isomorphism 
$$\fA(\cA'_{\gr,\tf,m}[n])\isoto\bigoplus_{k=0}^m\fA(\cA)$$
in $\cC$.\\
$\mathrm{(3)}$ 
Let $\cT$ be a triangulated category closed under countable coproducts 
and $\frakL\colon\RelEx_{\text{consist}} \to \cT$ a fine localizing theory. 
Then we have the canonical isomorphism 
between the polynomial object $\frakL(\cA)$ and $(\frakL(\cA'_{\gr}[n]),\frakL((-1)))$:
$$\lambda_{\cA,n}\colon 
\frakL(\cA)[t]\isoto (\frakL(\cA'_{\gr}[n]),\frakL((-1)))$$
in $\End\cT$ which makes the diagram in $\cT$ below commutative for any 
natural number $k$:
$$\footnotesize{\displaystyle{
\xymatrix{
\frakL(\cA) \ar[d]_{t^k} \ar[rd]^{\frakL(\cF[n](-k))}\\
\overset{\infty}{\underset{m=0}{\bigoplus}}\frakL(\cA)t^m 
\ar[r]^{\sim}_{\frakL(\lambda_{\cA,n})} & \frakL(\cA'_{\gr}[n]).
}}}$$
\end{thm}

\begin{proof}
$\mathrm{(1)}$ 
We apply Lemma~\ref{lem:resol thm} to $\cA'_{\gr}[n]$ and Koszul homologies. 
The assumption of Lemma~\ref{lem:resol thm} follows from 
Lemma~\ref{lem:fundamental facts about} $\mathrm{(2)}$ and 
Proposition~\ref{prop:fun pro Kos hom} $\mathrm{(1)}$.

\sn
$\mathrm{(2)}$ 
We define the exact functor $b\colon\cA^{\times m+1} \to \cA'_{\gr,\tf,m}[n]$ 
by sending an object ${(x_k)}_{0\leq k\leq m}$ in $\cA^{\times m+1}$ 
to $\displaystyle{\bigoplus_{k=0}^{m} \cF[n](x_k)(-k)}$ 
in $\cA'_{\gr,\tf,m}[n]$.
By virtue of Proposition~\ref{prop:fun pro Kos hom} $\mathrm{(1)}$, 
the functor $ab$ is canonically isomorphic to 
the identity functor on $\cA^{\times m+1}$. 
On the other hand, 
the identity functor on $\cA'_{\gr,\tf,m}[n]$ has 
an exact characteristic filtration $F_{\bullet}$ 
with $F_px/F_{p-1}x\isoto \cF[n]({T_0(x)}_p)(-p)$ 
for any object $x$ in $\cA'_{\gr,\tf,m}[n]$ 
by Proposition~\ref{prop:fun pro Kos hom} $\mathrm{(6)}$, 
so applying Lemma~\ref{lem:characteristic filtration}, 
we have the equalities 
$$\id_{\cA'_{\gr,\tf,m}[n]}=\sum_{p=1}^{m}\fA(F_p/F_{p-1})=
\sum_{p=0}^m\fA(\cF[n]({T_0(-)}_p)(-p))=\fA(ba).$$
Therefore we have 
an isomorphism 
$$\fA(\cA'_{\gr,\tf,m}[n])\isoto 
\displaystyle{\bigoplus_{i=0}^m}\fA(\cA).$$
$\mathrm{(3)}$ 
By Remark~\ref{rem:localizing theory} $\mathrm{(2)}$ $\mathrm{(i)}$, 
for any integer $m$, we have an isomorphism 
$$\frakL(\cA'_{\gr,\tf,m}[n])\isoto 
\displaystyle{\bigoplus_{i=0}^m}\frakL(\cA).$$
Finally by taking the filtered inductive limit 
and utilizing assertion $\mathrm{(1)}$ and 
Remark~\ref{rem:localizing theory} $\mathrm{(2)}$ $\mathrm{(i)}$, 
we get the desired isomorphism.
\begin{eqnarray*}
\bigoplus_{i=0}^{\infty}\frakL(\cA) & = & 
\underset{m\to \infty}{\colim} \bigoplus_{i=0}^{m}\frakL(\cA) \isoto 
\underset{m\to \infty}{\colim} \frakL(\cA'_{\gr,\tf,m}[n])\\ 
&\isoto& 
\frakL(\underset{m\to \infty}{\colim}\cA'_{\gr,\tf,m}[n]) \isoto 
\frakL(\cA'_{\gr,\tf}[n]) \isoto
\frakL(\cA'_{\gr}[n]).
\end{eqnarray*}
\end{proof}

\section{The main theorem}
\label{sec:Main thm}

In this section, 
let us fix an essentially small noetherian abelian category $\cA$. 
We consider the functor 
$-\otimes_{\cA}\bbZ[t]$ from $\cA$ to $\cA[t]$ 
defined by sending an object $a$ in $\cA$ to an object $a[t]$ in $\cA[t]$. 
Since $\Lex\cA$ is Gorthendieck, we can easily check that 
the functor $-\otimes_{\cA}\bbZ[t]$ is exact. 
Recall that we regard $\ExCat$ 
the category of essentially small exact categories as 
the full subcategory or $\RelEx_{\operatorname{vs,\ consisit}}$. 
(See Example~\ref{ex:semi devices} $\mathrm{(1)}$). 
The purpose of this section is to study the induced map 
from $-\otimes_{\cA}\bbZ[t]$ on $K$-theory. 
More generally, we will prove the following theorem:

\begin{thm}
\label{thm:abstract main thm}
Let $(\cT,\Sigma)$ be 
a triangulated category closed under countable coproducts, 
$\calR$ a full subcategory of $\RelEx_{\operatorname{consist}}$ which contains 
$\AbCat$ the category of essentially small abelian categories and 
$\frakL\colon\calR \to \cT$ a 
nilpotent invariant fine localizing theory. 
Then the base change functor 
$-\otimes_{\cA}\bbZ[t]\colon\cA\to \cA[t]$ 
induces an isomorphism 
$$\frakL(\cA)\isoto\frakL(\cA[t]).$$
\end{thm}

By taking $\cT$, $\calR$ and $\frakL$ to the 
stable category of spectra, $\ExCat$ the category of essentially small exact categories 
and the non-connective $K$-theory, 
we get Theorem~\ref{thm:mainthm} from Theorem~\ref{thm:abstract main thm}. 
From now on, let $(\cT,\Sigma)$ be a triangulated category closed under 
countable coproducts and $(\frakL,\partial)$ a fine locaizing theory 
$\frakL\colon\ExCat \to \cT$. 

\subsection{Nilpotent objects in $\cA'_{\gr}[2]$}
\label{subsec:Nilp}

In this subsection, 
we will define the category $\cA'_{\gr,\nil}[2]$ 
of nilpotent objects 
in $\cA'_{\gr}[2]$. 
We also study the relationship $\cA'_{\gr}[2]$ with 
$\cA[t]$ and calculate the $K$-theory of $\cA'_{\gr,\nil}[2]$. 
Recall from the introduction that 
Theorem~\ref{prop:cannonical isom} and 
Proposition~\ref{prop;devissage} 
correspond to geometric motivational formulas, namely 
the localization and 
the purity formulas 
in the introduction respectively. 
For simplicity in this subsection, 
we write $\psi$ and $\phi$ for $\psi^1$ and $\psi^2$ respectively 
and for any object $x$ in $\cA$ and 
we write $x[\psi,\phi]$ for $\cF[2](x)$. 

\begin{df}
\label{df:nilpotent obj}
Let $\cE$ be an exact category. 
An object $x$ in $\cE_{\gr}[2]$ is ({\bf $\psi$-}) {\it nilpotent} 
if there exists an integer $n$ such that 
$$\psi^{n}_k=0$$ 
for any non-negative integer $k$.
We write $\cE_{\gr,\nil}[2]$ 
(resp. $\cE'_{\gr,\nil}[2]$, $\cE_{\gr,\nil}[2]_{\fingen}$) 
for the full subcategory of 
$\cE_{\gr}[2]$ 
(resp. $\cE'_{\gr}[2]$, $\cE_{\gr}[2]_{\fingen}$) 
consisting of all 
nilpotent objects. 
\end{df}

\begin{lem}
\label{lem:niliSerresubcat}
The category $\cA'_{\gr,\nil}[2]$ 
is a Serre subcategory of $\cA'_{\gr}[2]$. 
In particular $\cA'_{\gr,\nil}[2]$ 
is an abelian category.
\end{lem}

\begin{proof}
The assertion that $\cA'_{\gr,\nil}[2]$ 
is closed under 
sub- and quotient objects and finite direct sum is easily proved. 
We can also easily prove the following assertion. 
For a short exact sequence 
$x \rinf y \rdef z$ in $\cA'_{\gr}$, 
let $i$ and $j$ be integers such that 
$\psi_x^i=0$ and $\psi^j_z=0$. 
Then we can easily prove that $\psi^{i+j}_y=0$. 
Therefore $\cA'_{\gr,\nil}[2]$ 
is closed under extensions in $\cA'_{\gr}[2]$. 
\end{proof}

\begin{df}
\label{nt:thetabar}
Let $\cE$ be an essentially small exact category. 
We define the functor 
$$\bar{\Theta}_{\cE}(=\bar{\Theta})\colon \cE_{\gr}[2] \to \End\Lex \cE$$ 
which sends an object $x$ in $\cE_{\gr}[2]$ to an object 
$(\underset{\psi}{\colim}x_n,\colim\phi_n )$ 
in $\End\Lex \cE$ where 
$\underset{\psi}{\colim} x_n$ is an inductive limit of an ind system 
$(x_0 \onto{\psi_0} x_1 \onto{\psi_1} x_2 \onto{\psi_2}\cdots)$, 
namely $\displaystyle{\coker\left (\bigoplus_{n=0}^{\infty}x_n 
\onto{\id-\oplus\psi_n} 
\bigoplus_{n=0}^{\infty}x_n\right )}$ 
and $\colim \phi_n$ is an inductive limit of $\{\phi_n\}_n$, 
namely, 
a morphism which is induced from 
$\displaystyle{\bigoplus_{n=0}^{\infty}\phi_n}$. 
\end{df}

\begin{lem}
\label{lem:thetabar}
Let $\cE$ be an essentially small exact category. 
Then\\
$\mathrm{(1)}$ 
The functor 
$\bar{\Theta}_{\cE}\colon\cE_{\gr}[2] \to \End \Lex\cE$ is an exact functor. 
Moreover if $u:x\to y$ is an epimorphism in $\cE_{\gr}[2]$, 
then $\Theta(u)\colon\Theta(x) \to \Theta(y)$ 
is also an epimorphism in $\End\Lex\cE$.\\
$\mathrm{(2)}$ 
For any object $x$ in $\cE_{\gr,\nil}[2]$, 
$\bar{\Theta}(x)$ is a zero object.\\
$\mathrm{(3)}$ 
For any object $x$ in $\cE$ and any positive integer $k$, 
$\psi^k\colon x(-k) \to x$ induces an isomorphism 
$\bar{\Theta}(\psi^k)\colon\Theta(x(-k))\to \Theta(x)$ in $\End\Lex\cE$.\\
$\mathrm{(4)}$ 
For any object $x$ in $\cE_{\gr}[2]$ and any positive integer $k$, 
$\bar{\Theta}(x[\psi,\phi](-k))$ is canonically 
isomorphic to $x[t]$.\\
$\mathrm{(5)}$ 
For any object $x$ in $\cE_{\gr}[2]_{\fingen}$, 
$\bar{\Theta}(x)$ is in ${(\End\Lex\cE)}_{\fingen}$. 
We denote the induced functor 
$\cE_{\gr}[2]_{\fingen} \to {(\End\Lex\cE)}_{\fingen}$ 
by $\bar{\Theta}_{\cE,\fingen}$.\\
In particular $\bar{\Theta}_{\cA,\fingen}$ induces the exact functor 
$\Theta_{\cA}(=\Theta)\colon \cA'_{\gr}[2]/\cA'_{\gr,\nil}[2] \to \cA[t]$. 
\end{lem}

\begin{proof}
$\mathrm{(1)}$ 
The functor $\bar{\Theta}_{\cE}$ factors through 
the functor $y_{\gr}[2]\colon \cE_{\gr}[2] \to {(\Lex \cE)}_{\gr}[2]$ which is 
induced from the yoneda embedding $y:\cE \to \Lex\cE$ and 
the colimit functor $\colim_{\psi}\colon {(\Lex\cE)}_{\gr}[2] \to \End\Lex\cE$. 
Obviously $y_{\gr}[2]$ is exact and preserves epimorphisms. 
Since $\Lex\cE$ is a Grothendieck category, 
the functor $\colim_{\bbN}\colon \HOM(\bbN,\Lex\cE) \to \Lex\cE$ is exact. 
In particular, we acquire the assertion that 
the functor $\bar{\Theta}_{\cE}$ is 
an exact functor and preserves epimorphisms. 

\sn
$\mathrm{(2)}$ 
For any object $x$ in $\cE_{\gr,\nil}[2]$, 
assume that $\psi^{m,k}_x=0$ for any non-negative integer $k$. 
Then $\displaystyle{\sum_{i=0}^{m-1}\psi_x^i}$ 
is the inverse morphism of $\displaystyle{\id-\bigoplus_{n=0}^{\infty}\psi_n}$. 
Therefore $\displaystyle{\bar{\Theta}(x)=
\coker\left (\bigoplus_{n=0}^{\infty}x_n \onto{\id-\oplus \psi_n} 
\bigoplus_{n=0}^{\infty}x_n\right )}$ 
is trivial.

\sn
$\mathrm{(3)}$ 
Obviously $\ker(\psi^k:x(-k) \to x)$ and $\coker(\psi^k:x(-k)\to x)$ 
are $\psi$-nilpotent in $\Lex\cE_{\gr}[2]$. 
Therefore $\bar{\Theta}$ induces an isomorphism 
$\bar{\Theta}(\psi^k)$ by the observation in the proof of $\mathrm{(2)}$. 

\sn
$\mathrm{(4)}$ 
By assertion $\mathrm{(3)}$, we shall assume that $k=0$. 
In this case we have the canonical isomorphisms 
$$\bar{\Theta}(x[\psi,\phi])\isoto 
\coker\left (\bigoplus_{n=0}^{\infty} \bigoplus_{i+j=n}x\psi^i\phi^j 
\onto{\id-\oplus\psi_n} \bigoplus_{n=0}^{\infty} 
\bigoplus_{i+j=n}x\psi^i\phi^j\right )\isoto 
\bigoplus_{n=0}^{\infty}x\phi^n$$
where $x\psi^i\phi^j$ and $x\phi^n$ are copies of $x$. 

\sn
$\mathrm{(5)}$ 
For any object $x$ in $\cE_{\gr}[2]_{\fingen}$, 
there exists a non-negative integer $n$ such that 
the canonical morphism 
$\displaystyle{\bigoplus_{k=0}^n x_k[\psi,\phi](-k) \rdef x}$ 
is an epimorphism in $\cE_{\gr}[2]$. 
Then by $\mathrm{(1)}$ and $\mathrm{(4)}$, 
we have an epimorphism 
$\displaystyle{\bigoplus_{k=0}^n x_k[t] \rdef \bar{\Theta}(x)}$ 
in $\End\Lex\cE$. 
Therefore by Remark~\ref{rem:other df of S-poly cat}, 
$\bar{\Theta}(x)$ is in ${(\End\Lex\cE)}_{\fingen}$. 
\end{proof}

\begin{thm}
\label{prop:cannonical isom} 
The functor 
$\Theta\colon \cA'_{\gr}[2]/\cA'_{\gr,\nil}[2]\to \cA[t]$ 
is an equivalence of categories. 
\end{thm}

To prove Theorem~\ref{prop:cannonical isom}, 
we need to the following lemmata: 

\begin{lem}
\label{lem:module case faithful}
Let $R$ be a ring with unit and let us consider 
the polynomial ring $R[t]$ over $R$ and 
let 
$\displaystyle{M=\bigoplus_{n=0}^{\infty} M_n}$ 
be a finitely generated graded right $R[t]$-module. 
If the map $1-t:M \to M$ is surjective, 
then $M$ is $t$-nilpotent. 
Namely, there exists an integer $n$ such that 
$Mt^n=0$. 
\end{lem}

\begin{proof}
Since $M$ is finitely generated by homogenious elements, 
we shall just check that for any homogenious element $y$ in $M_k$, 
there exists a positive integer $l$ such that $yt^l=0$. 
By assumption, there exists 
an element $\displaystyle{x=\sum_{j=0}^mx_j}$ in $M$ 
such that we have the equality 
\begin{equation}
\label{equ:(1-t)x=y}
x(1-t)=y
\end{equation}
where $x_j$ is the $j$th homogenious component of $x$. 
By comparing the homogenious components 
of the equality (\ref{equ:(1-t)x=y}), 
we notice that $x_j$ is equal to $0$ if $0\leq j \leq k-1$ or $j=m$, 
$y$ if $j=k$ and $x_{j-1}t$ if $k+1\leq j\leq m-1$. 
Therefore if $m\leq k$, we have $y_k=0$ and if $m>k$, 
we have $yt^{m-k}=x_{k+1}t^{m-k-1}=\cdots =x_m=0$. 
Hence we get the desired result.
\end{proof}

\begin{para}
\label{para:faithful proof}
We prove that $\Theta$ is faithful. 
By Corollary~\ref{cor:embedd thm}, 
there exists a ring with unit $R$ and 
an exact fully faithful embeddings $j\colon \cA\rinc \cM_R$ 
and $k\colon \Lex\cA \rinc \Mod(R)$ which makes the diagram below commutative: 
$${\footnotesize{\xymatrix{
\cA \ar[r]^{y_{\cA}} \ar[d]_{j} & \Lex\cA \ar[d]^{k}\\
\cM_R \ar[r]_{\iota} & \Mod(R)
}}}$$
where the functor $y_{\cA}$ is the yoneda embedding functor and 
$\iota$ is the canonical inclusion functor. 
Then the functor $j$ induces the 
fully faithful embedding
$$j'\colon =j_{\gr}[2]_{\fingen}\colon 
\cA_{\gr}[2]_{\fingen}=\cA'_{\gr}[2] \rinc 
{(\cM_R)}_{\gr}[2]_{fg}=\cM_{R[t_1,t_2]\gr}$$ 
which makes the diagram below commtative 
by virtue of Remark~\ref{rem:finitely generated} and 
Example~\ref{ex:gradedcategory}. 
$${\footnotesize{\xymatrix{
\cA_{\gr}[2] \ar[r]^{\!\!\!\! j_{\gr}[2]_{\fingen}} \ar[d]_{\bar{\Theta}_{\cA,\fingen}} & 
\cM_{R[t_1,t_2],\gr}\ar[d]^{\bar{\Theta}_{\cM_R,\fingen}}\\
\cA[t]={(\End\Lex\cA)}_{\fingen} \ar[r]_{\ \ \ \ \ \End(k)} & \End\Mod(R) .
}}}$$
For an object $x$ in $\cA'_{\gr}[2]$, 
assume that $\bar{\Theta}_{\cA,\fingen}(x)$ is a zero object. 
Then by Lemma~\ref{lem:module case faithful}, 
$j'(x)$ is a $t_1$-nilpotent $R$-module and 
therefore $x$ is $\psi$-nilpotent. 
Hence $\Theta_{\cA}$ is faithful. 
\qed 
\end{para}

\begin{df}[\bf $\psi$-free object]
\label{nt:psi free} 
An object $z$ in $\cA_{\gr}[2]$ is {\it $\psi$-free} 
if a morphism $\psi_n\colon z_n \to z_{n+1}$ is a monomorphism 
for any non-negative integer $n$. 
\end{df}

\begin{lem}
\label{lem:psi-free}
For any object $y$ in $\cA'_{\gr}[2]$, 
there exists a $\psi$-free object $z$ in $\cA'_{\gr}[2]$ 
and an epimorphism $u\colon y \to z$ in $\cA'_{\gr}[2]$ 
such that the object $\ker u$ is in $\cA_{\gr,\nil}'[2]$. 
\end{lem}

\begin{proof}
For any non-negative integer $n$, 
we denote the canonical morphism 
from $y_n$ to $\underset{\phi}{\colim} y_j=\theta(y)$ by 
$\iota_n:y_n \to \theta(y)$ and we put $z_n:=\im \iota_n$. 
Then we have the commmutative diagrams below
$${\footnotesize{\xymatrix{
y_n \ar[r]^{\psi_n} \ar[d]_{\iota_n} & y_{n+1} \ar[d]^{\iota_{n+1}}\\
\Theta(y) \ar[r]_{\colim\psi_j} & \Theta(y)  
}
\xymatrix{
y_n \ar[rr]^{\phi_n} \ar[rd]_{\iota_n} & & y_{n+1} \ar[ld]^{\iota_{n+1}}\\
& \Theta(y) .
}}}$$
Therefore $\psi_n$ 
and $\phi_n$ 
induce a morphism $z_n \onto{\bar{\psi}_n} z_{n+1}$ 
and a monomorphism $z_n \overset{\bar{\phi}_n}{\rinc} z_{n+1}$ 
for any non-negative integer $n$. 
Then $z=\{z_n,\bar{\psi}_n,\bar{\phi}_n\}$ is a $\psi$-free object 
in $\cA_{\gr}[2]$ and there exists a canonical short exact sequence 
$$\ker \mu \rinf y \overset{\mu}{\rdef} z.$$
Notice that $y$ is in $\cA'_{\gr}[2]$ and 
therefore $z$ is also in $\cA'_{\gr}[2]$. 
Obviously $\Theta(y) \onto{\Theta(\mu)} \Theta(z)=\Theta(y)$ is 
an isomorphism in $\cA[t]$. 
Hence by \ref{para:faithful proof}, 
the object $\ker \mu$ is in $\cA'_{\gr,\nil}[2]$. 
\end{proof}

\begin{lem}
\label{lem:key lem for full}
$\mathrm{(1)}$ 
For any object $x$ in $\cA$, 
any $\psi$-free object $y$ in $\cA'_{\gr}[2]$ and 
any morphism $\Theta(x[\psi,\phi])=x[t] \onto{a} \Theta(y)$, 
there exists a non-negative integer $k$ and 
a morphism $u\colon x[\psi,\phi](-k) \to y$ in $\cA'_{\gr}[2]$ 
such that $a=\Theta(x[\psi,\phi] \overset{\psi^k}{\leftarrow} 
x[\psi,\phi](-k) 
\onto{u}y)$.\\
$\mathrm{(2)}$ 
For any $\psi$-free object $y$ and any object $z$ in $\cA_{\gr}[2]$ 
and any morphism $\Theta(z) \onto{a} \Theta(y)$, 
there exists a non-negative integer $k$ and a morphism 
$u\colon z_n[\psi,\phi](-k) \to y$ in $\cA'_{\gr}[2]$ 
such that $\Theta(u)$ makes the diagram below commutative 
$${\footnotesize{\xymatrix{
& \Theta(z_n[\psi,\phi](-n-k)) \ar[r]^{\Theta(\alpha(-k))} 
\ar[d]^{\Theta(\psi^k)} \ar[ld]_{\Theta(\psi^n)} & 
\Theta(z(-k)) \ar[d]^{\Theta(\psi^k)}\\
\Theta(z_n[\psi,\phi](-k)) \ar[rd]_{\Theta(u)} & 
\Theta(z_n[\psi,\phi](-n)) \ar[r]^{\Theta(\alpha)} & \Theta(z) \ar[ld]^a\\
&\Theta(y)
}}}$$
where the morphism $\alpha\colon z_n[\psi,\phi](-n) \to z$ 
is the canonical morphism as in Remark~\ref{rem:adjointnessofF[n]}. 
\end{lem}

\begin{proof}
$\mathrm{(1)}$ 
We denote the composition of 
morphisms 
$x \to x[t] \onto{a} \Theta(y)$ 
in $\Lex\cA$ 
by $\bar{a}$ and the canonical morphism from $y_n$ to $\Theta(y)$ 
by $\iota_n:y_n\to \Theta(y)$ for any non-negative integer $n$. 
Since $\im\bar{a}$ is a quotient of $x$ in $\Lex\cA$, 
it is noetherian by Lemma~\ref{lem:noetherian} $\mathrm{(1)}$ 
and therefore an asscending chain of subobjects of $\im\bar{a}$, 
$\{\im \bar{a}\cap \im\iota_n\}_{n\in\bbN}$ is stational, 
say $\im \bar{a}\cap\im\iota_k=\im\bar{a}\cap\iota_{k+1}=\cdots$. 
Then since $\Lex\cA$ is Grothendieck, 
we have $\im \bar{a}=\colim_{i\geq k}\im\bar{a}\cap\im\iota_i=
\im\bar{a}\cap \im\iota_k$. 
Therefore the morphism $\bar{a}$ 
factors through morphisms 
$x\onto{a'} y_k$ and $y_k \onto{\iota_k} \Theta(y)$. 
By Remark~\ref{rem:adjointnessofF[n]}, $a'$ induces the desired morphism 
$u:x[\psi,\phi](-k) \to y$. 

\sn
$\mathrm{(2)}$ 
By applying assertion $\mathrm{(1)}$ to the morphism 
$a\Theta(\alpha):z_n[t] \to \Theta(y)$, 
we get the assertion.
\end{proof}

\begin{para}
We prove that $\Theta$ is full. 
Namely, for any objects $x$, $y$ in $\cA'_{\gr}[2]$, 
we prove that the map
$$\Theta\colon\Hom_{\cA_{\gr}'[2]/\cA_{\gr,\nil}'[2]}(x,y) 
\to \Hom_{\cA[t]}(\Theta(x),\Theta(y))$$ 
is surjective. 
By Lemma~\ref{lem:psi-free}, 
we may assume that $y$ is $\psi$-free. 
By Lemma~\ref{lem:fundamental facts about} $\mathrm{(2)}$, 
there exists a non-negative integer $m$ 
such that the canonical morphism $\displaystyle{z\colon =
\bigoplus_{k=0}^m x_k[\psi,\phi](-k) \overset{P}{\rdef} x}$ 
is an epimorphism. 
Let $u\colon\Theta(x) \to \Theta(y)$ be a morphism in $\cA[t]$. 
Then by Lemma~\ref{lem:key lem for full} $\mathrm{(2)}$, 
there exists a non-negative integer $l$ and a morphism 
$u\colon z(-l) \to y$ which makes the right diagram below commutative 
$${\footnotesize{\xymatrix{
\ker P(-l) \ar@{>->}[r]^j & z(-l) \ar[d]_u \ar@{->>}[r]^{P(-l)} & x(-l) 
\ar@{-->}[ld]^{\bar{u}}\\
& y
}
\xymatrix{
\Theta(\ker P(-l)) \ar@{>->}[r]^{\Theta(j)} & \Theta(z(-l)) 
\ar[d]_{\Theta(u)} \ar@{->>}[r]^{\Theta(P(-l))} & \Theta(x(-l)) 
\ar[d]^{\Theta(\psi^l)}\\
& \Theta(y) & \Theta(x) \ar[l]^a .
}}}$$
Since $\Theta$ is faithful, $uj$ is the zero morphism, 
$u$ induces a morphism $\bar{u}\colon x(-l) \to y$ 
in the left commutative diagram above and we have 
the equality $a=\Theta(x\overset{\psi^l}{\leftarrow}x(-l) \onto{\bar{u}}y)$. 
Hence we get the desired result. 
\qed
\end{para}

\begin{cor}
\label{cor:canisom} 
The sequence 
$$\cA'_{\gr,\nil}[2] \to \cA'_{\gr}[2] \onto{\bar{\Theta}} \cA[t].$$
is derived exact. 
In particular, there exists a distingushed triangle 
$$\frakL(\cA'_{\gr,\nil}[2])\to 
\frakL(\cA'_{\gr}[2]) \onto{\frakL(\bar{\Theta})}
\frakL(\cA[t]) \onto{\partial} 
\Sigma\frakL(\cA'_{\gr,\nil}[2])$$
in $\cT$.
\end{cor}

\begin{proof}
We check the condition $\mathrm{(\ast)}$ in Example~\ref{ex:exact seq of ab cat}.
Let $y\rinf x$ be a monomorphism in $\cA'_{\gr}[2]$ with $y$ 
in $\cA'_{\gr,\nil}[2]$. 

\begin{claim}
There exists an integer $n\geq 0$ such that 
$y\cap \im(\psi^n\colon x(-n)\to x)$ is trivial.
\end{claim}

If we prove the claim, then the composition 
$y\rinf x \rdef \coker(\psi^n\colon x(-n)\to x)$ is the desired monomorphism. 

\begin{proof}[Proof of Claim]
Consider the faithful exact functor 
$\Gamma\colon \cA'_{\gr}[2] \to \cA[t_1,t_2]$, 
$x\mapsto \left (\bigoplus x_n,\bigoplus \phi_n,\bigoplus \psi_n \right )$ 
in the proof of Lemma~\ref{lem:fundamental facts about} $\mathrm{(1)}$. 
Then by the abstract Artin-Rees lemma~\ref{cor:abst AR}, 
there exists an intger $n_0\geq 0$ such that 
\begin{equation}
\label{equ:AR-lem}
\im(t_{2,\Gamma(x)}^n\colon \Gamma(x)\to \Gamma(x))\cap \Gamma(y)=
\im(t_{2,\Gamma(x)}^{n-n_0}\colon (\im(t^{n_0}_{2,\Gamma(x)}\colon \Gamma(x)\to\Gamma(x))\cap\Gamma(y)) \to \Gamma(y))
\end{equation}
for any $n\geq n_0$. 
Moreover since $y$ is a $\psi$-nilpotent object, the right hand side 
of the equality $\mathrm{(\ref{equ:AR-lem})}$ is trivial for sufficiently large 
$n$. 
By the faithfulness of $\Gamma$, 
we obtain the result.
\end{proof}  
By Example~\ref{ex:exact seq of ab cat} and 
Theorem~\ref{prop:cannonical isom}, the sequence 
$\cA'_{\gr,\nil}[2]\to \cA'_{\gr}[2] \onto{\bar{\Theta}}\cA[t]$ 
is derived exact.
\end{proof}

Recall the definition of Serre radical from Definition~\ref{df:Serre radical}.

\begin{prop}
\label{prop;devissage}
We regard $\cA'_{\gr}[1]$ as a full subcategory of $\cA'_{\gr,\nil}[2]$ 
by the exact functor defined by sending 
an object $(x,\psi^1)$ in $\cA'_{\gr}[1]$ to 
an object $(x,0,\psi^1)$ in $\cA'_{\gr,\nil}[2]$. 
Then $\cA'_{\gr}[1]$ is closed under taking 
finite direct sums, 
admissible sub- and quotient objects in $\cA'_{\gr,\nil}[2]$ 
and ${}^S\!\!\!\sqrt{\cA'_{\gr}[1]}=\cA'_{\gr,\nil}[2]$. 
In particular the inclusion functor $\cA'_{\gr}[1]\rinc \cA'_{\gr,\nil}[2]$ 
induces an isomorphism $\frakL(\cA'_{\gr}[1])\isoto\frakL(\cA'_{\gr,\nil}[2])$. 
\end{prop}

\begin{proof}
Obviously 
$\cA'_{\gr}[1]$ is closed under 
taking finite direct sums, 
admissible sub- and quotient objects 
in $\cA'_{\gr,\nil}[2]$. 
Moreover for any $x$ in $\cA'_{\gr,\nil}[2]$, 
let us consider the filtration $\{\im\psi^k\}_{k\in\bbN}$ of $x$. 
Then for each $k$, $\im\psi^k/\im\psi^{k+1}$ is 
isomorphic to an object in $\cA'_{\gr}[1]$. 
The last assertion follows from niloptent invariance of $\frakL$.
\end{proof}

For an object $x$ in an additive category $\cB$, 
recall the definition 
of the polynomial object $x[t]$ in $\End\cB$ from \ref{para:polynomial category}.

\begin{cor}
\label{cor:devissage}
We have the canonical isomorphism between the polynomial object $\frakL(\cA)[t]$ 
of $\frakL(\cA)$ 
and $\frakL(\cA'_{\gr,\nil}[2])$:
$$\lambda'\colon\frakL(\cA)[t]
\isoto \frakL(\cA'_{\gr,\nil}[2]).$$
Here for any non-negative integer $m$, 
$\frakL(\cA)t^m\to \frakL(\cA'_{\gr,\nil}[2])$ 
is induced from an exact functor 
$\cA\to \cA'_{\gr,\nil}[2]$ which sends an object $a$ in $\cA$ to 
$(a[\phi](-m),0,\phi)$ in $\cA'_{\gr,\nil}[2]$.
\qed
\end{cor}

\subsection{The proof of the main theorem}
\label{subsec:The proof of the main theorem}

In this subsection, 
we will finish the proof of Theorem~\ref{thm:abstract main thm}. 
For an object $x$ in an additive category $\cB$, 
recall the definition 
of the polynomial object $x[t]$ in $\End\cB$ from \ref{para:polynomial category}. 
The key lemma is the following:

\begin{lem}
\label{lem:commutative diagram}
There exists the commutative diagram below
$$\xymatrix{
\frakL(\cA)[t] \ar[r]^{\lambda'} \ar[d]_{\id-t} & \frakL(\cA'_{\gr,\nil}[2]) 
\ar[d]\\
\frakL(\cA)[t] \ar[r]_{\lambda_{\cA,2}} & \frakL(\cA'_{\gr}[2])
}$$
where $\frakL(\cA)[t]$ is the polynomial object of $\frakL(\cA)$ and 
the right vertical morphism are induced from the 
inculsion functor $\cA'_{\gr,\nil}[2]\rinc \cA'_{\gr}[2]$.
\end{lem}

\begin{proof}
For each $k\geq 0$, 
we consider the commutativity of diagram below
\begin{equation}
\label{equ:dia}
\xymatrix{
\frakL(\cA)t^k
\ar[r] \ar[d]_{\id-t} & \frakL(\cA'_{\gr,\nil}[2]) 
\ar[d]\\
\frakL(\cA)[t] \ar[r]_{\lambda_{\cA,2}} & \frakL(\cA'_{\gr}[2]).
}
\end{equation}
An object $a$ in $\cA$ goes to $(a[\phi](-k),0,\phi)$ by 
the compositions of the functors 
$\cA \to \cA'_{\gr,\nil}[2] \to \cA'_{\gr}[2]$ and 
notice that the functor $\cF[2](-k):\cA\to\cA'_{\gr}[2]$ induces 
$\frakL(\cA) \onto{t^k} 
\frakL(\cA)[t]\isoto \frakL(\cA'_{\gr}[2])$ 
by Theorem~\ref{cor:canisomoftf}. 
On the other hand, 
for any object $a$ in $\cA$, 
there exists an exact sequence in $\cA'_{\gr}[2]$ 
$$a[\phi,\psi](-k-1) \overset{\psi}{\rinf} a[\phi,\psi](-k) \rdef 
(a[\phi](-k),0,\phi).$$
By Lemma~\ref{lem:characteristic filtration}, 
this implies that the diagram $\mathrm{(\ref{equ:dia})}$ 
is commutative. 
\end{proof}

\begin{proof}[Proof of Theorem~\ref{thm:abstract main thm}]
Consider the following commutative diagram
$$\xymatrix{
\frakL(\cA)[t] \ar[r]^{\id-t} \ar[d]^{\wr}_{\lambda'} & 
\frakL(\cA)[t] \ar[r] \ar[d]^{\wr}_{\lambda_{\cA,2}} & 
\frakL(\cA) \ar[d]^{\frakL(-\otimes_{\cA}\bbZ[t])}\\
\frakL(\cA'_{\gr,\nil}[2]) \ar[r] & 
\frakL(\cA'_{\gr}[2]) \ar[r]_{\frakL(\bar{\Theta})} 
& \frakL(\cA[t]) \ar[r]_{\partial} & \Sigma\frakL(\cA'_{\gr,\nil}[2]).
}$$
Since top horizontal line is a split exact sequence 
by Lemma~\ref{lem:split exact seq} 
and $\lambda_{\cA,2}$ and $\lambda'$ are isomorphisms by 
Theorem~\ref{cor:canisomoftf} and Corollary~\ref{cor:devissage}, 
the bottom distingushed triangle is also split, namely $\partial=0$. 
Hence $\frakL(-\otimes_{\cA}\bbZ[t])$ is an isomorphism by the five lemma.
\end{proof}

\subsection{Generalized Vorst problem}
\label{subsec:Gen Vorst prob}

In this subsection, 
we propose the geenralized Vorst problem. 
We start by defining the notion of regularity for abelian categories. 
Recall the definition of 
homological dimensions of abelian categoies from \ref{rem:GMnotation}.

\begin{df}[\bf Regular abelian category]
\label{df:regular abel}
Let $\cB$ be an abelian category.\\
$\mathrm{(1)}$ 
We denote the full subcategory of projective objects in $\cB$ 
by $\cP_{\cB}$.\\
$\mathrm{(2)}$ 
$\cB$ is {\it regular} if it is noetherian and 
the inclusion functor $\cP_{\cB}\rinc \cB$ is a derived equivalence. 
The last condition is equivalent to the condition that 
$\hdim\cB$ is finite. 
\end{df}

\begin{prop}
\label{prop:GM} 
{\rm (\cf \cite[III 5.16, 5.19]{GM96}).} 
Let $\cA$ be a noetherian abelian category.\\
$\mathrm{(1)}$ 
For any projective object $p$ in $\cA$, 
$p[t]$ is also a projective object in $\cA[t]$. 
In particular, there exists the base change functor 
$-\otimes_{\cA}\bbZ[t]\colon\cP_{\cA} 
\to \cP_{\cA[t]}$ which sending an object $p$ to $p[t]$.\\
$\mathrm{(2)}$ 
If $\cA$ is regular, then $\cA[t]$ is also regular. 
\end{prop}

\begin{proof}
$\mathrm{(1)}$ is proven in \cite[III 5.19]{GM96}. 
$\mathrm{(2)}$ follows from the proof of \cite[III 5.20]{GM96} and 
Theorem~\ref{thm:Abst Hilb basis}.
\end{proof}

\begin{cor}
\label{cor:mainthm}
Let $\cT$ be a triangulated category, 
$\frakL:\ExCat \to \cT$ a nipotent invariant localizing theory 
and $\cA$ a regular noetherian essentially small abelian category. 
Then the base change functor 
$-\otimes_{\cA}\bbZ[t]\colon\cP_{\cA} \to \cP_{\cA[t]}$ 
induces an isomorphism 
$$\frakL(\cP_{\cA})\isoto\frakL(\cP_{\cA[t]}).$$
\end{cor}

\begin{proof}
The inclusion functors $\cP_{\cA}\rinc \cA$ and $\cP_{\cA[t]}\rinc \cA[t]$ 
induce the commutative diagram below:
$$\xymatrix{
\frakL(\cP_{\cA}) \ar[r]^{\frakL(-\otimes_{\cA}\bbZ[t])} \ar[d] & 
\frakL(\cP_{\cA[t]}) \ar[d]\\
\frakL(\cA) \ar[r]_{\frakL(-\otimes_{\cA}\bbZ[t])} & 
\frakL(\cA[t]).
}$$
Here the vertical morphisms and the bottom horizontal morphism are isomorphisms 
by regularity of $\cA$, Proposition~\ref{prop:GM} $\mathrm{(2)}$ and 
Theorem~\ref{thm:abstract main thm} respectively. 
Hence we obtain the result.
\end{proof}

\begin{prob}[\bf Generalized Vorst problem]
\label{prob:GVP}
Is the converse of Corollary~\ref{cor:mainthm} true or not? 
More precisely, let $\cA$ be a noetherian abelian category which has 
enough projective objects. 
Assume that for any positive integer $r$, 
the base change functor $\cP_{\cA} \to \cP_{\cA[t_1,\cdots,t_r]}$ 
induces an isomorphism on (connective) $K$-theory:
$$K(\cP_{\cA})\isoto K(\cP_{\cA[t_1,\cdots,t_r]}).$$
Then is $\cA$ regular or not?
\end{prob}

The problem has an affirmative answer if 
$\cA$ is a category of finitely generated modules over 
a noetherian commutative ring essentially of 
finite type over a base field and if 
we assume the resolution of singularities. 
(See \cite{CHW08} and \cite{GH12}). 

\sn
\textbf{Acknowledgements.} 
The authors wish to express their deep gratitude 
to Marco Schlichting 
for instructing them in the proof of 
the abstract Hilbert basis theorem~\ref{thm:Abst Hilb basis}. 
They also very thank to the referee for giving valuable comments.

\mn
SATOSHI MOCHIZUKI\\
{\it{DEPARTMENT OF MATHEMATICS,
CHUO UNIVERSITY,
BUNKYO-KU, TOKYO, JAPAN.}}\\
e-mail: {\tt{mochi@gug.math.chuo-u.ac.jp}}\\

\mn
AKIYOSHI SANNAI\\
{\it{GRADUATE SCHOOL OF MATHEMATICAL SCIENCES,
UNIVERSITY OF TOKYO,
3-8-1 KOMABA, MEGURO-KU, TOKYO, JAPAN.}}\\
e-mail: {\tt{sannai@ms.u-tokyo.ac.jp}}
\end{document}